\documentclass[a4paper,11pt]{article}
\usepackage{xypic}
\usepackage{mathtools}
\usepackage{amsmath}
\usepackage{mathrsfs}
\usepackage{amssymb}
\usepackage{amsthm}
\usepackage{dsfont}
\usepackage{graphicx}
\usepackage{bmpsize}
\usepackage{times}
\usepackage{ textcomp }
\usepackage{color}
\usepackage{mathtools}
\usepackage{latexsym, empheq, fancybox}
\usepackage{mathrsfs}
\usepackage{exscale}
\usepackage{booktabs, array}
\usepackage[english]{babel}
\newcommand{\nc}{\newcommand}
\nc{\rnc}{\renewcommand}
\numberwithin{equation}{section}
\newtheorem{theorem}{Theorem}[section]
\newcounter{abc}
\newtheorem*{thm}{\stepcounter{abc}Theorem \Alph{abc}}
\newtheorem{lemma}[theorem]{Lemma}

\newtheorem{definition}[theorem]{Definition}

\newtheorem{corollary}[theorem]{Corollary}

\newtheorem{example}[theorem]{Example}
\newtheorem{remark}[theorem]{Remark}

\usepackage{authblk}

\newcommand{\R}{\mathbb{R}}
\newcommand{\N}{\mathbb{N}}
\newcommand{\Z}{\mathbb{Z}}
\newcommand{\To}{\rightarrow}

\newcommand{\FOC}{\mathcal{FOC}}
\newcommand{\U}{\mathcal{U}}
\newcommand{\V}{\mathcal{V}}
\newcommand{\W}{\mathcal{W}}
\newcommand{\Or}{\mathcal{O}}
\newcommand{\PO}{\mathcal{PO}_S}
\newcommand{\eps}{\epsilon}

\newcommand{\diam}{\text{diam}}

\newcommand{\ilim}{\underset{\longleftarrow}{\lim}}
\newcommand{\nempty}{\neq\emptyset}

\title{\Large\textbf{Shadowing and mixing on systems of countable group actions}
}

\author{
		Zijie Lin
	\thanks{School of Mathematical Sciences and Institute of Mathematics, Nanjing Normal University, Nanjing 210023, P. R. China
		(E-mail: zjlin137@126.com)},
		Ercai Chen
	\thanks{School of Mathematical Sciences and Institute of Mathematics, Nanjing Normal University, Nanjing 210023, P. R. China,
		and Center  of Nonlinear Science, Nanjing University, Nanjing 210093, P. R.China
		(E-mail: ecchen@njnu.edu.cn),
		},
		Xiaoyao Zhou
		\thanks{School of Mathematical Sciences and Institute of Mathematics, Nanjing Normal University, Nanjing 210023, P. R. China
			(E-mail: zhouxiaoyaodeyouxian@126.com),}
}

\begin{document}
	\date{}
	\maketitle
	
\begin{abstract}
	Let $(X,G,\Phi)$ be a dynamical system, where $X$ is a compact Hausdorff space, and $G$ is a countable discrete group. We investigated shadowing property and mixing between subshifts and general dynamical systems. For the shadowing property, fix some finite subset $S\subset G$. We proved that if $X$ is totally disconnected and $\Phi$ has $S$-shadowing property, then $(X,G,\Phi)$ is conjugate to an inverse limit of a sequence of shifts of finite type which satisfies Mittag-Leffler condition. Also, suppose that $X$ is a metric space (may be not totally disconnected), we proved that if $\Phi$ has $S$-shadowing property, then $(X,G,\Phi)$ is a factor of an inverse limit of a sequence of shifts of finite type by a factor map which almost lifts pseudo-orbit for $S$.
	
	On the other hand, let property $P$ be one of the following properties: transitivity, minimal, totally transitivity, weakly mixing, mixing, and specification property. We proved that if $X$ is totally disconnected, then $\Phi$ has property $P$ if and only if $(X,G,\Phi)$ is conjugate to an inverse limit of an inverse system that consists of subshifts with property $P$ which satisfies Mittag-Leffler condition. Also, for the case of metric space (may be not totally disconnected), if property $P$ is not minimal or specification property, we proved that $\Phi$ has property $P$ if and only if $(X,G,\Phi)$ is a factor of an inverse limit of a sequence of subshifts with property $P$ which satisfies Mittag-Leffler condition.
\end{abstract}
	
\section{Introduction}
A \emph{dynamical system} is a triple $(X,G,\Phi)$,
where $X$ is a compact Hausdorff space,
$G$ is a countable discrete group with the identity $e_G$ and $\Phi: G\times X\To X$ satisfies that $\Phi_g:=\Phi(g,\cdot)$ is a homeomorphism,
$\Phi_{e_G}(x)=x$ and $\Phi_{gg'}(x)=\Phi_g(\Phi_{g'}(x))$ for any $g,g'\in G$ and any $x\in X$,
where $e_G$ is the identity of $G$.
A group $G$ is said to be \emph{finitely generated} if there exists a finite set $S\subset G$ such that for any $g\in G$, there exist $s_1,\dots,s_n\in S\cup S^{-1}$ such that $g=s_1\cdots s_n$. In this case, $S$ is called a finite \emph{generator} of $G$. A generator $S$ is called \emph{symmetric} if $S^{-1}\subset S$.

For a dynamical system $(X,G,\Phi)$, suppose that $X$ is a compact metric space with metric $d$.
For $x\in X$ and $\eps>0$, denote $B(x,\eps)=\{y\in X:\, d(x,y)<\eps\}$.

For the case $G=\Z$, the shadowing property in dynamical systems has been studied for several years. See \cite{Pal00}, \cite{Pil99} and \cite{Pil11} for details. In \cite{Bow75}, the shadowing lemma plays a key role in the theory of shadowing, which shows that in a small neighbourhood of a hyperbolic set, the shadowing property is established. For the case $G=\Z^d$, S. Yu. Pilyugin and S. B. Tikhomirov introduce the notion of shadowing for $\Z^n$ and $\Z^n\times \R^m$ in \cite{PT03}. It is shown in \cite{Kos07} that generic $\Z^2$-actions of an interval have the periodic shadowing property and strong tolerance stability. In \cite{Opro08,Opr08}, for topologically Anosov $\Z^d$-actions that are not topologically hyperbolic in any direction, an analogue of spectral decomposition theorem is proved. The relations between shadowing property and distality for actions of $\R^n$ is studied in \cite{KK12}.

Motivated by the notion of shadowing for $\Z^n$ and $\Z^n\times \R^m$ in \cite{PT03}, the authors (\cite{OT14}) introduced the shadowing property for the finitely generated group, which is not necessarily abelian groups. They show that for a finitely generated nilpotent group, an action of the whole group has the shadowing property (and expansiveness) if the action of at least one element has the shadowing property and expansiveness.

Furthermore, shadowing property for a countable group is defined in \cite{CL18} and \cite{Mey19}. We used the definition of shadowing property for a countable group in \cite{CL18} and \cite{Mey19}, which generalize the definition for finitely generated group in \cite{Bar14} and \cite{OT14}. For $S\subset G$ and $\delta>0$, \emph{$(S,\delta)$-pseudo-orbit} of $(X,G,\Phi)$ is a sequence $\{x_g\}_{g\in G}$ in $X$ satisfying that $d(\Phi_s(x_g),x_{sg})<\delta$ holds for any $g\in G$ and $s\in S$.
For $\eps>0$, a point $z\in X$ is said to \emph{$\eps$-shadow} a sequence $\{x_g\}_{g\in G}$ in $X$ if $d(\Phi_g(z),x_g)<\eps$ for any $g\in G$.
For $S\subset G$, we say $\Phi$ has \emph{$S$-shadowing property} if for any $\eps>0$, there exists $\delta>0$ such that each $(S,\delta)$-pseudo-orbit can be $\eps$-shadowed by some point in $X$.


In \cite{GM20}, C. Good and J. Meddaugh give a topological definition of shadowing property on a dynamical system $(X,f)$ where $X$ is compact Hausdorff space and $f:X\To X$ is a continuous map.
They prove that it is equivalent to the classic shadowing property when $X$ is a compact metric space.
For a finite open cover $\U$, a sequence $\{x_n\}_{n\in\N}$ is called a \emph{$\U$-pseudo-orbit} if for any $n\in\N$ there exists $U_n\in\U$ such that $x_{n+1},f(x_n)\in U_n$.
And a point $z\in X$ is called \emph{$\U$-shadow} a sequence $\{x_n\}_{n\in\N}$ if for any $n\in\N$ there exists $U_n\in\U$ such that $f^n(z),x_n\in U_n$.
The map $f$ is said to have the \emph{shadowing property} if for any finite open cover $\U$ of $X$, there exists finite open cover $\V$ such that each $\V$-pseudo-orbit can be $\U$-shadowed by some point in $X$.
C. Good and J. Meddaugh prove the following results in \cite{GM20}.
\begin{thm}[\cite{GM20}]
	Let $f:X\To X$ be a continuous map where $X$ is a compact totally disconnected Hausdorff space. Then $f$ has shadowing property if and only if it is conjugate to an inverse limit of an ML inverse system consisting of shifts of finite type.
\end{thm}

\begin{thm}[\cite{GM20}]
	Let $f:X\To X$ be a continuous map where $X$ is compact metric space. If $f$ has shadowing property, then it is a factor of an inverse limit of a sequence of shifts of finite type.
\end{thm}

Naturally, we consider whether their results hold for the case of group actions.
We also give a topological definition of $S$-shadowing property and prove the following results.

\begin{theorem}\label{t:1.1}
	Suppose that $(X,G,\Phi)$ is a dynamical system, where $X$ is compact totally disconnected Hausdorff space. Let $S$ be a finite subset of $G$. If $\Phi$ has $S$-shadowing property, then $(X,G,\Phi)$ is conjugate to the inverse limit of an ML inverse system of shifts of finite type over $S\cup \{e_G\}$. And if $(X,G,\Phi)$ is conjugate to the inverse limit of an ML inverse system of subshifts which have $S$-shadowing property, then $\Phi$ has $S$-shadowing property.
\end{theorem}

If $G$ is a finitely generated group and $S$ is a finite generator of $G$, then it is proved in \cite[Theorem 3.2]{CL18} that a subshift is shift of finite type if and only if it has $S$-shadowing property. Thus we have the following corollary.

\begin{corollary}
	Suppose that $(X,G,\Phi)$ is a dynamical system, where $X$ is compact totally disconnected Hausdorff space and $G$ is a finitely generated group. Let $S$ be a finite generator of $G$. Then $\Phi$ has $S$-shadowing property if and only if $(X,G,\Phi)$ is conjugate to the inverse limit of an ML inverse system of shifts of finite type.
\end{corollary}

Also, for the metric case, we prove the following result.

\begin{theorem}\label{t:1.2}
	Let $(X,G,\Phi)$ be a dynamical system, where $X$ is compact metric systems, and $S$ be a finite subset of $G$. If $\Phi$ has $S$-shadowing property, then there exists an inverse system $(\phi^{n+1}_n,(X_n,G,\sigma))$ consisting of shifts of finite type over $S\cup \{e_G\}$ such that $(X,G,\Phi)$ is a factor of $(\ilim\{\phi^{n+1}_n,X_n\},G,\sigma^*)$.
\end{theorem}

Inspired by the results that systems with shadowing property can be represented by shifts of finite type which are subshifts with shadowing property, we consider other dynamical properties that will have similar results.
So we get the following theorems, which show that subshifts can be seen as fundamental objects in the theory of mixing.

\begin{theorem}\label{t:1.3}
	Suppose that $(X,G,\Phi)$ is a dynamical system, where $X$ is a compact totally disconnected Hausdorff space. Let property $P$ be one of the following properties: transitivity, minimal,  totally transitivity, weakly mixing, mixing, and specification property. Then $\Phi$ has property $P$ if and only if $(X,G,\Phi)$ is conjugate to the inverse limit of an ML inverse system of subshifts with property $P$.
\end{theorem}

\begin{theorem}\label{t:1.4}
	Suppose that $(X,G,\Phi)$ is a dynamical system, where $X$ is compact metric space. Let property $P$ be one of the following properties: transitivity, totally transitivity, weakly mixing, and mixing. Then $\Phi$ has property $P$ if and only if $(X,G,\Phi)$ is a factor of the inverse limit of an ML inverse system of subshifts with property $P$.
\end{theorem}

\begin{remark}
	For the case of $\N$, that is, dynamical system $(X,f)$ where $X$ is compact Hausdorff space and $f:X\To X$ is a continuous surjection, we can also prove similar results of Theorem \ref{t:1.3} and Theorem \ref{t:1.4}.
\end{remark}

This paper is organized as follows.
In Section 2, we give some denotations and notations.
In Section 3, we give the topological definition of $S$-shadowing property for compact Hausdorff space and show that $S$-shadowing property is preserved for inverse limits. Finally, we give the definition of orbit spaces and pseudo-orbit spaces for countable group actions.
Section 4 is the proof of Theorem \ref{t:1.1} and Theorem \ref{t:1.2}.
In Section 5, we show a class of factor maps which preserve shadowing property.
In Section 6, we give the proof of Theorem \ref{t:1.3} and Theorem \ref{t:1.4}.


\section{Preliminaries}
In this section, we give some basic denotations and notations.

For a finite set $\Sigma$ and a countable group $G$, the dynamical system $(\Sigma^G,G,\sigma)$ is called a \emph{full shift with alphabet $\Sigma$ over group $G$}, where $\Sigma^G$ is endowed by the Tychonoff topology, and for each $g\in G$, $\sigma_g(x)(h)=x(hg)$ for any $h\in G$.
If we take $G=\{g_0=e_G,g_1,g_2,\dots\}$, a metric $d$ on $\Sigma^G$, which is compatible with the Tychonoff topology, can be taken as:
$$d(x,y)=2^{-\min\{n\in\N:x(g_n)\neq y(g_n)\}}$$
for any $x,y\in \Sigma^G$.
For a set $A\subset G$ and $P_A\in\Sigma^A$, denote the cylinder set by
$$C(P_A)=\{y\in \Sigma^G:y(g)=P_A(g),\text{ for any }g\in A\}.$$
A \emph{subshift} is a closed $G$-invariant subset $X$ of $\Sigma^G$.
For a finite subset $A\subset G$, the \emph{language} $\mathcal{L}_A(X)$ of a subshift $(X,G,\sigma)$ is the set $\{P_A\in\Sigma^A:C(P_A)\cap X \nempty\}$.
For a finite subset $A$ of $G$, the subshift $X$ is called \emph{shift of finite type over $A$} if there exists a collection of \emph{forbidden patterns} $\mathcal{F}\subset \Sigma^A$ such that
$$X=\{x\in\Sigma^G:\text{for any }g\in G \text{ and any }P_A\in\mathcal{F},\,\sigma_g(x)\notin C(P_A)\}.$$

For two dynamical systems $(X,G,\Phi^X)$ and $(Y,G,\Phi^Y)$, recall that $(Y,G,\Phi^Y)$ is a \emph{factor} of $(X,G,\Phi^X)$ if there exists a continuous surjection $\phi:X\To Y$ satisfying $\Phi^Y_g\circ\phi=\phi\circ\Phi^X_g$ for any $g\in G$ and the map $\phi$ is called a \emph{factor map} from $(X,G,\Phi^X)$ to $(Y,G,\Phi^Y)$. Moreover, we call $(X,G,\Phi^X)$ is \emph{conjugate} to $(Y,G,\Phi^Y)$ if the factor map $\phi$ is injective.

A pair $(\Lambda,\le)$ is called a \emph{directed} set if $\le$ is a transitive order and for any $\lambda,\eta\in\Lambda$ there exists $\gamma\in\Lambda$ such that $\lambda\le\gamma$, $\eta\le\gamma$. A subset $A$ of $\Lambda$ is \emph{cofinal} if for any $\lambda\in\Lambda$, there exists $\eta\in A$ such that $\lambda\le\eta$.

\begin{definition}
	Let $(\Lambda,\le)$ be a directed set. For each $\lambda\in\Lambda$, let $X_\lambda$ be a compact Hausdorff space, and for each $\lambda\le\eta$, $\phi^{\eta}_{\lambda}:X_\eta\To X_\lambda$ is a continuous map. Then $(\phi^{\eta}_{\lambda},X_\lambda)$ is called \emph{inverse system} if
	\begin{enumerate}
		\item[(1)] $\phi^{\lambda}_{\lambda}$ is the identity map;
		\item[(2)] For $\lambda\le\eta\le\gamma$, $\phi^{\gamma}_{\lambda}=\phi^{\eta}_{\lambda}\circ\phi^{\gamma}_{\eta}$.
	\end{enumerate}

	The \emph{inverse limit} of $(\phi^{\eta}_{\lambda},X_\lambda)$ is the space
	$$\ilim\{\phi^{\eta}_{\lambda},X_\lambda\}=\{(x_\lambda)_{\lambda\in\Lambda}\in \prod_{\lambda\in\Lambda}X_\lambda:\text{for any }\lambda\le\eta,\,x_\lambda=\phi^{\eta}_{\lambda}(x_{\eta})\},$$
	whose topology is inherited by the Tychonoff topology on $\prod_{\lambda\in\Lambda}X_\lambda$.
	
	Also, the inverse system $(\phi^{\eta}_{\lambda},X_\lambda)$ satisfies \emph{Mittag--Leffler condition} (or ML condition for short) if for any $\lambda\in\Lambda$, there exists $\eta\ge\lambda$ such that for any $\gamma\ge\eta$, $\phi^\eta_\lambda(X_\eta)=\phi^\gamma_\lambda(X_\gamma)$. In this case, the inverse system is called \emph{ML inverse system}.
\end{definition}

It can be seen that the inverse limit of $(\phi^{\eta}_{\lambda},X_\lambda)$ is a compact Hausdorff space and $\{\pi_\lambda^{-1}(U_\lambda)\cap\ilim\{\phi^{\eta}_{\lambda},X_\lambda\}:\lambda\in\Lambda\text{ and } U_\lambda \text{ is an open subset of }X_\lambda\}$ forms a basis for $\ilim\{\phi^{\eta}_{\lambda},X_\lambda\}$, where $\pi_\lambda$ is the projection from the product space $\prod_{\eta\in\Lambda}X_\eta$ to space $X_\lambda$.
If the inverse system $(\phi^{\eta}_{\lambda},X_\lambda)$ satisfies the ML condition and $\gamma$ witnesses the condition with respect to $\mu$, then $\pi_\mu^{-1}(x)\cap\ilim\{\phi^{\eta}_{\lambda},X_\lambda\}\nempty$ for any $x\in\phi^\gamma_\mu(X_\gamma)$. Also, if the bonding maps $\phi^\eta_\lambda$ is surjective, then the inverse system satisfies the ML condition.

Now, suppose that there is a group $G$ which actions on $X_\lambda$ by $\Phi^\lambda$ for each $\lambda\in\Lambda$, that is, $(X_\lambda,G,\Phi^\lambda)$ is dynamical system for each $\lambda\in\Lambda$. If the bonding maps $\phi^\eta_\lambda:X_\eta\To X_\lambda$ commute with the action $\Phi^\lambda$, we can also extend this definition to the family of dynamical systems $\{(X_\lambda,G,\Phi^\lambda):\lambda\in\Lambda\}$.

\begin{definition}
	Let $(\Lambda,\le)$ be a directed set. For each $\lambda\in\Lambda$, let $(X_\lambda,G,\Phi^\lambda)$ be a dynamical system, and for each $\lambda\le\eta$, $\phi^{\eta}_{\lambda}:X_\eta\To X_\lambda$ is a continuous map. Then $(\phi^{\eta}_{\lambda},(X_\lambda,G,\Phi^\lambda))$ is called \emph{inverse system} if
	\begin{enumerate}
		\item[(1)] $\phi^{\lambda}_{\lambda}$ is the identity map;
		\item[(2)] For $\lambda\le\eta\le\gamma$, $\phi^{\gamma}_{\lambda}=\phi^{\eta}_{\lambda}\circ\phi^{\gamma}_{\eta}$;
		\item[(3)] For $\lambda\le\eta$ and $g\in G$, $\Phi^\lambda_g\circ\phi^{\eta}_{\lambda}=\phi^{\eta}_{\lambda}\circ\Phi^\eta_g$.
	\end{enumerate}
	
	The \emph{inverse limit} of $(\phi^{\eta}_{\lambda},(X_\lambda,G,\Phi^\lambda))$ is the dynamical system
	$(\ilim\{\phi^{\eta}_{\lambda},X_\lambda\},G,\Phi^*)$
	where $\Phi^*_g$ is given by
	$$\Phi^*_g((x_\lambda)_{\lambda\in\Lambda})=(\Phi^\lambda_g(x_\lambda))_{\lambda\in\Lambda}$$
	for any $g\in G$.
\end{definition}

\section{Shadowing in compact Hausdorff space}
In this section, inspired by the idea in \cite{GM20}, we investigated a topological definition of shadowing for group action.
The metric version of shadowing for group action is introduced in \cite{Mey19}, and in \cite{CL18}, it is the case of finitely generated groups.

To define the topological version of shadowing property, we need some notations and definitions.
For a topological space $X$, denote by $\FOC(X)$ the collection of all finite open covers of $X$. Here, we always assume that each $\U\in\FOC(X)$ does not contain empty set.

\begin{definition}
	Let $(X,G,\Phi)$ be a dynamical system, $S$ is a subset of $G$, and $\U\in\FOC(X)$.
	\begin{itemize}
		\item[(1)] The sequence $\{x_g\}_{g\in G}$ on $X$ is called \emph{$(S,\U)$-pseudo-orbit} if there exists a sequence $\{U_g\}_{g\in G}$ on $\U$ such that for any $g\in G$ and $s\in S$, $x_{sg},\Phi_s(x_g)\in U_{sg}$.
		
		\item[(2)] For a sequence $\{U_{g}\}_{g\in G}$ on $\U$, it is called \emph{$(S,\U)$-pseudo-orbit pattern} if there exists $(S,\U)$-pseudo-orbit such that $x_{sg},\Phi_s(x_g)\in U_{sg}$ for all $g\in G$ and $s\in S$. And it is called \emph{$\U$-orbit pattern} if there exists a point $x\in X$ such that $\Phi_g(x)\in U_g$ for all $g\in G$.
		
		\item[(3)] We say a point $x\in X$ \emph{$\U$-shadows} $\{x_g\}_{g\in G}$ if for each $g\in G$, there exists $U_g\in\U$ such that $\Phi_g(x),x_g\in U_g$.
	\end{itemize}
\end{definition}

If a sequence $\{x_g\}_{g\in G}$ on $X$ satisfies that for any $s\in S$ and $g\in G$, there exists $U_{s,g}\in \U$ such that $x_{sg},\Phi_s(x_g)\in U_{s,g}$, then it may not correspond any $(S,\U)$-pseudo-orbit pattern, noting that it may happen $U_{s,g}\neq U_{s',g'}$ with some $sg=s'g'$.

\begin{example}\label{ex}
	For the dynamical system $(\{0,1\}^{\Z^2},\Z^2,\sigma)$, we take a finite subset $S$ of $\Z^2$ is the set $\{\vec{e}_1=(1,0),\vec{e}_2=(0,1)\}$. Let $A=\{(0,0),(1,0)\}$, and for any $a,b\in\{0,1\}$, take $P^{(a,b)}_A$ such that $P^{(a,b)}_A((0,0))=a$ and $P^{(a,b)}_A((1,0))=b$. Set $U_1=C(P^{(0,0)}_A)\cup C(P^{(1,0)}_A)$, $U_2=C(P^{(0,0)}_A)\cup C(P^{(0,1)}_A)$,  $U_3=C(P^{(1,1)}_A)$ and $\U=\{U_1,U_2,U_3\}$ be a finite open cover of $\{0,1\}^{\Z^2}$.
	Take a sequence $\{x_{\vec{a}}\}_{\vec{a}\in \Z^2}$ such that $\sigma_{(0,1)}(x_{(1,0)})\in C(P^{(1,0)}_A)$, $\sigma_{(1,0)}(x_{(0,1)})\in C(P^{(0,1)}_A)$, $x_{(1,1)}\in C(P^{(0,0)}_A)$, and satisfies that for any $\vec{s}\in S$ and $\vec{a}\in \Z^2$, there exists $U_{\vec{s},\vec{a}}\in \U$ such that $x_{\vec{s}+\vec{a}},\sigma_{\vec{s}}(x_{\vec{a}})\in U_{\vec{s},\vec{a}}$. Noticing that $\sigma_{(0,1)}(x_{(1,0)}),x_{(1,1)}\in U_1$ and $\sigma_{(1,0)}(x_{(0,1)}),x_{(1,1)}\in U_2$, the sequence $\{x_{\vec{a}}\}_{\vec{a}\in \Z^2}$ does not respond to any $(S,\U)$-pseudo-orbit pattern.
\end{example}

However, if the elements of $\U$ are pairwise disjoint, then for some $sg=s'g'$, $x_{sg}=x_{s'g'}\in U_{s,g}\cap U_{s',g'}$, which implies that $U_{s,g}=U_{s',g'}$. Thus, in this case, such a sequence $\{x_g\}_{g\in G}$ is $(S,\U)$-pseudo-orbit if for any $s\in S$ and $g\in G$, there exists $U_{s,g}\in \U$ such that $x_{sg},\Phi_s(x_g)\in U_{s,g}$.

\begin{lemma}
	Let $(X,G,\Phi)$ be a dynamical system, where $X$ is compact metric space. Let $S$ be a subset of $G$. Then $\Phi$ has $S$-shadowing property if and only if for any $\U\in\FOC(X)$, there exists $\V\in\FOC(X)$ such that each $(S,\V)$-pseudo-orbit can be $\U$-shadowed by some point in $X$.
\end{lemma}

\begin{proof}
	First, suppose that $\Phi$ has $S$-shadowing property.
	For any $\U\in\FOC(X)$, take $\eps$ be the Lebesgue number of $\U$.
	By the $S$-shadowing property of $\Phi$, there exists $\delta>0$ such that each $(S,\delta)$-pseudo-orbit can be $\eps$-shadowed by some point in $X$.
	By the compactness of $X$, take $\V$ be a finite subcover of $\{B(x,\delta/2): x\in X\}$. For any $(S,\V)$-pseudo-orbit $\{x_g\}_{g\in G}$, it can be seen that $d(\Phi_s(x_g),x_{sg})<\delta$, which implies it is $(S,\delta)$-pseudo-orbit. So there exists a point $z\in X$ such that $d(\Phi_g(z),x_g)<\eps$ for all $g\in G$. It means that there exists $U_g\in \U$ such that $\Phi_g(z),x_g\in U_g$, that is, $z$ $\U$-shadows $\{x_g\}_{g\in G}$.
	
	Conversely, for any $\eps>0$, we take $\U$ be a finite subcover of $\{B(x,\eps/2): x\in X\}$. So there exists $\V\in\FOC(X)$ such that each $(S,\V)$-pseudo-orbit can be $\U$-shadowing by some point in $X$. Take $\delta$ be the Lebesgue number of $\V$. Fix any $(S,\delta)$-pseudo-orbit $\{x_g\}_{g\in G}$. For any $g\in G$, there exists $V_g
	\in\V$ such that $B(x_g,\delta)\subset V_g$. If $g=s'g'$ for some $s'\in S$ and $g'\in G$, then $\Phi_{s'}(x_{g'})\in B(x_g,\delta)\subset V_g$. So $\{x_g\}_{g\in G}$ is $(S,\V)$-pseudo-orbit with $(S,\V)$-pseudo-orbit pattern $\{V_g\}_{g\in G}$. Thus there exists a point $z\in X$ $\U$-shadowing $\{x_g\}_{g\in G}$, which implies that $z$ $\eps$-shadows $\{x_g\}_{g\in G}$. Thus $\Phi$ has $S$-shadowing property.
\end{proof}

By the above lemma, for a dynamical system $(X,G,\Phi)$ where $X$ is a compact Hausdorff space, we say that $\Phi$ has \emph{$S$-shadowing property} if, for any $\U\in\FOC(X)$, there exists $\V\in\FOC(X)$ such that each $(S,\V)$-pseudo-orbit can be $\U$-shadowed by some point in $X$. Since $(S,\V)$-pseudo-orbit is equivalent to $(S\cup \{e_G\},\V)$-pseudo-orbit, $S$-shadowing property is equivalent to $S\cup \{e_G\}$-shadowing property. For two covers $\U,\V\in\FOC(X)$, recall that $\V$ \emph{refines} $\U$ (denote by $\V\succ\U$) if for any $V\in\V$, there exists $U\in\U$ such that $V\subset U$.

The following theorem shows that the shadowing property is preserved by inverse limit.

\begin{theorem}\label{t:3.4}
	Let $S$ be a subset of $G$. Suppose that dynamical system $(X,G,\Phi)$ is conjugate to an inverse limit $(\ilim\{\phi^{\eta}_{\lambda},X_\lambda\},G,(\Phi^\lambda)^*)$ satisfying the ML condition. If $\Phi^\lambda$ has $S$-shadowing property for each $\lambda\in\Lambda$, then $\Phi$ has $S$-shadowing property.
\end{theorem}

\begin{proof}
	Without loss of generality, let $(\Lambda,\le)$ be a directed set and $(X,G,\Phi)$ be the inverse limit $(\ilim\{\phi^{\eta}_{\lambda},X_\lambda\},G,(\Phi^\lambda)^*)$ satisfying the ML condition, where $\Phi^\lambda$ has $S$-shadowing property for each $\lambda\in\Lambda$.
	
	Fix any $\U\in\FOC(X)$.
	Since the set $$\{\pi_\lambda^{-1}(U_\lambda)\cap X:\lambda\in\Lambda\text{ and } U_\lambda \text{ is an open subset of }X_\lambda\}$$
	forms a basis for $\ilim\{\phi^{\eta}_{\lambda},X_\lambda\}$ and $(\Lambda,\le)$ is a directed set, there exist $\lambda\in\Lambda$ and a finite open cover $\W_\lambda$ of $X_\lambda$ such that $\{\pi^{-1}_\lambda(W)\cap X:W\in\W_\lambda\}$ is a finite open cover of $X$ refining $\U$.
	By the ML condition, there exists $\eta\ge\lambda$ such that $\phi^\gamma_\lambda(X_\gamma)=\phi^\eta_\lambda(X_\eta)$ for any $\gamma\ge\eta$.
	Let $\W_\eta=\{(\phi^\eta_\lambda)^{-1}(W): W\in\W_\lambda\}$.
	Since $\Phi^\eta$ has $S$-shadowing property, we can find $\V_\eta\in\FOC(X_\eta)$ witnesses the property with respect to $\W_\eta$ and let $\V=\{\pi^{-1}_\eta(V)\cap X: V\in\V_\eta\}$.
	
	We claim that each $(S,\V)$-pseudo-orbit can be $\U$-shadowed by some point in $X$.
	Fix any $(S,\V)$-pseudo-orbit $\{x_g\}_{g\in G}$ with pattern $\{\pi^{-1}_\eta(V_g)\cap X\}_{g\in G}$.
	Noticing that $\Phi^\eta_s(\pi_\eta(x_g)),$ $\pi_\eta(x_{sg})\in V_{sg}$ for any $g\in G$ and any $s\in S$, it means that $\{\pi_\eta(x_g)\}_{g\in G}$ is $(S,\V_\eta)$-pseudo-orbit with pattern $\{V_g\}_{g\in G}$.
	Because $\V_\eta$ witnesses the $S$-shadowing property with respect to $\W_\eta$, there exists a point $z_\eta\in X_\eta$ which $\W_\eta$-shadows $\{\pi_\eta(x_g)\}_{g\in G}$, that is, for each $g\in G$ there exists $W_g\in\W_\eta$ such that $\Phi^\eta_g(z_\eta), \pi_\eta(x_g)\in W_g=(\phi^\eta_\lambda)^{-1}(W'_g)$ for some $W'_g\in\W_\lambda$.
	Thus by the ML condition, for $z_\lambda=\phi^\eta_\lambda(z_\eta)\in\phi^\eta_\lambda(X_\eta)$, there exists a point $z\in X$ with $\pi_\lambda(z)=z_\lambda$. For each $g\in G$, we have $\Phi^\lambda_g\pi_\lambda(z)=\phi^\eta_\lambda(\Phi^\eta_g(z_\eta)), \pi_\lambda(x_g)=\phi^\eta_\lambda(\pi_\eta(x_g))\in W'_g\in\W_\lambda$, which means that $z$ $\W_\lambda$-shadows $\{x_g\}_{g\in G}$. Since $\W_\lambda$ refines $\U$, it ends the proof.
\end{proof}

Inspired by \cite{GM20}, we also consider the notion of orbit space and pseudo-orbit for countable group actions.

\begin{definition}
	Let $(X,G,\Phi)$ be a dynamical system and $S$ be a subset of $G$, $\U\in\FOC(X)$, and $\U^G$ be the full shift with alphabet $\U$ over group $G$.
	\begin{itemize}
		\item[(1)] The \emph{$\U$-orbit space} is the set $\Or(\U)$ which is the closure of the set
		$$\{\{U_g\}_{g\in G}\in\U^G:\bigcap_{g\in G}\Phi_{g^{-1}}(U_g)\nempty\}.$$
		
		\item[(2)] The \emph{$(S,\U)$-pseudo-orbit space} is the set
		$$\PO(\U)=\{\{U_g\}_{g\in G}\in\U^G:\text{there is $(S,\U)$-pseudo-orbit with pattarn } \{U_g\}_{g\in G}\}.$$
	\end{itemize}
\end{definition}

For each $g\in G$, denote by $\pi_g$ the projection from $\U^G$ to $\U$ onto the coordinate $g$.
The following lemma shows that both orbit space and pseudo-orbit space are subshifts and, in particular, pseudo-orbit space is a shift of finite type.

\begin{lemma}\label{l:3.6}
	Let $(X,G,\Phi)$ be a dynamical system, $S$ be a finite subset of $G$, and $\U\in\FOC(X)$.
	Then $\Or(\U)$ and $\PO(\U)$ are closed $G$-invariant subset of $\U^G$, and
	$$\PO(\U)=\{\{U_g\}_{g\in G}\in\U^G:U_g\cap\bigcap_{s\in S}\Phi_{s^{-1}}(U_{sg})\nempty\text{ for any }g\in G\}.$$
	Both $\Or(\U)$ and $\PO(\U)$ are subshifts of $\U^G$ and $\PO(\U)$ is a shift of finite type over $S\cup \{e_G\}$.
\end{lemma}

\begin{proof}
	It is not hard to see that $\Or(\U)$ and $\PO(\U)$ is $G$-invariant and
	$$\PO(\U)\subset\{\{U_g\}_{g\in G}\in\U^G:U_g\cap\bigcap_{s\in S}\Phi_{s^{-1}}(U_{sg})\nempty\text{ for any }g\in G\}.$$
	For any sequence $\{U_g\}_{g\in G}$ satisfying that $U_g\cap\bigcap_{s\in S}\Phi_{s^{-1}}(U_{sg})\nempty$ for any $g\in G$. Fix any $x_g$ in $U_g\cap\bigcap_{s\in S}\Phi_{s^{-1}}(U_{sg})$ for each $g\in G$. So for any $s\in S$ and $g\in G$, $x_{g}\in U_g\cap\bigcap_{s'\in S}\Phi_{s'^{-1}}(U_{s'g})\subset \Phi_{s^{-1}}(U_{sg})$, which implies that $\{x_g\}_{g\in G}$ is $(S,\U)$-pseudo-orbit with pattern $\{U_g\}_{g\in G}$.
	
	By the equality, it implies that $\PO(\U)$ is closed. Also, observing that the collection of forbidden patterns $\mathcal{F}=\{P\in\U^{\{e_G\}\cup S}: P(e_G)\cap\bigcap_{s\in S}\Phi_{s^{-1}}(P(s))=\emptyset\}$ is finite, $\PO(\U)$ is shift of finite type over $S\cup \{e_G\}$.
\end{proof}

In the general case, it can be seen that  $$\PO(\U)\subsetneqq\{\{U_g\}_{g\in G}\in\U^G:\Phi_s(U_g)\cap U_{sg}\nempty\text{ for any }g\in G,s\in S\},$$
by noticing that if $\Phi_s(U_g)\cap U_{sg}\nempty$ and $\Phi_{s'}(U_g)\cap U_{s'g}\nempty$ for some $g,s\neq s'$, it may happen that $U_g\cap \Phi_{s^{-1}}(U_{sg})\cap \Phi_{s'^{-1}}(U_{s'g})=\emptyset$, which implies that it does not correspond to any $(S,\U)$-pseudo-orbit.

For some $\mathcal{A}\subset \U\in\FOC(X)$, let $\bigcup\mathcal{A}:=\bigcup_{A\in\mathcal{A}}A$.
The following theorem shows that the dynamics of $(X,G,\Phi)$ can be encoded by a proper systems of covers.

\begin{theorem}\label{t:3.7}
	Let $(X,G,\Phi)$ be a dynamical system and $\{\U_\lambda\}_{\lambda\in\Lambda}$ be a cofinal directed subset of $\FOC(X)$. Then for any $x\in X$, and for any choice $U_\lambda(x)\in\U_\lambda$ such that $x\in U_\lambda(x)$ for each $\lambda\in\Lambda$, we have $\{x\}=\bigcap_{\lambda\in\Lambda}U_\lambda(x)$, and for each $g\in G$,
	$$\{\Phi_g(x)\}=\bigcap_{\lambda\in\Lambda}\bigcup\pi_{e_G}\left(\sigma_g\left(\Or(\U_\lambda)\cap\pi_{e_G}^{-1}(U_\lambda(x))\right)\right).$$
\end{theorem}

\begin{proof}
	First, it is easy to see that $x\in \bigcap_{\lambda\in\Lambda}U_\lambda(x)$. If there is another point $y\neq x$ with $y\in\bigcap_{\lambda\in\Lambda}U_\lambda(x)$, we can find an open cover $\{U,V\}$ such that $x\in U\setminus V$ and $y\notin U$ since $X$ is a compact Hausdorff space. Because  $\{\U_\lambda\}_{\lambda\in\Lambda}$ is a cofinal directed subset of $\FOC(X)$, there exists $\lambda\in\Lambda$ such that $\U_\lambda\succ \{U,V\}$, which implies that $U_\lambda(x)\subset U$ and $y\notin U_\lambda(x)$. Therefore, $\{x\}=\bigcap_{\lambda\in\Lambda}U_\lambda(x)$.
	
	For any $g\in G$, fix any $z\in\bigcap_{\lambda\in\Lambda}\bigcup\pi_e\left(\sigma_g\left(\Or(\U_\lambda)\cap\pi_{e_G}^{-1}(U_\lambda(x))\right)\right)$. Then for any $\lambda\in\Lambda$, there exist $x_\lambda\in U_\lambda(x)$ and $U^{(\lambda)}_g\in\U_\lambda$ such that $\Phi_g(x_\lambda),z\in U^{(\lambda)}_g$. Thus $z\in\bigcap_{\lambda\in\Lambda}U^{(\lambda)}_g$, and $\Phi_g(x_\lambda)$ converges to $z$. Since $x_\lambda$ converges to $x$, we have $\Phi_g(x)=z$. It ends the proof by the fact that $\Phi_g(x)\in\bigcap_{\lambda\in\Lambda}\bigcup\pi_{e_G}\left(\sigma_g\left(\Or(\U_\lambda)\cap\pi_{e_G}^{-1}(U_\lambda(x))\right)\right)$.
\end{proof}

\section{$S$-Shadowing and shifts of finite type}

\newcommand{\Part}{\mathrm{Part}}

In this section, for the dynamical system $(X,G,\Phi)$, we consider two cases: the space $X$ is compact totally disconnected Hausdorff space or compact metric space. We always assume that $S$ is a finite subset of $G$.

\subsection{Shadowing in totally disconnected spaces}

For a totally disconnected space, the finite open covers of $X$ can be taken as clopen \emph{partitions}, that is, all the elements are closed, open and pairwise disjoint. Denote by $\Part(X)$ the collection of all the finite clopen partitions of $X$. So $\Part(X)$ is a cofinal directed subset of $\FOC(X)$ for the order $\prec$.
For two open covers $\U,\V$ of $(X,G,\Phi)$ with $\U\prec\V$, define a map $\iota:\V\To\U$ such that $V\cap\iota(V)\nempty$. If $\U,\V\in\Part(X)$, it means that $V\subset\iota(V)$ and $\iota$ is well defined. Also, we can define $\iota^G:\V^G\To\U^G$ by $\iota^G(\{V_g\}_{g\in G})=\{\iota(V_g)\}_{g\in G}$. We also use $\iota$ to denote $\iota^G$ for convenience if there is no confusion.

If $\U\in\Part(X)$, then we have
$$\Or(\U)=\{\{U_g\}_{g\in G}\in\U^G:\bigcap_{g\in G}\Phi_{g^{-1}}(U_g)\nempty\}.$$

\begin{lemma}
	Let $(X,G,\Phi)$ be a dynamical system, where $X$ is compact totally disconnected Hausdorff space, and $S$ be a finite subset of $G$. For $\U,\V\in\Part(X)$ with $\V\succ\U$, $\sigma$ and $\iota$ commute, and
	\begin{enumerate}
		\item[(1)] $\iota(\Or(\V))=\Or(\U)$;
		\item[(2)] $\Or(\U)\subset\iota(\PO(\V))\subset\PO(\U)$.
	\end{enumerate}
\end{lemma}

\begin{proof}
	It is obvious that $\sigma$ and $\iota$ commute by the definition of $\iota$.
	
	To prove (1), fix any $\{V_g\}_{g\in G}\in\Or(\V)$. Then there exists $x\in X$ such that $\Phi_g(x)\in V_g$ for any $g\in G$. Since $V_g\subset\iota(V_g)$, we have $\bigcap_{g\in G}\Phi_{g^{-1}}(\iota(V_g))\nempty$, which implies that $\iota(\{V_g\}_{g\in G})\in\Or(\U)$. For the converse, fix any $\{U_g\}_{g\in G}\in\Or(\U)$. Then there exists $x\in X$ such that $\Phi_g(x)\in U_g$ for any $g\in G$. Also, for any $g\in G$, there exists $V_g\in\V$ such that $\Phi_g(x)\in V_g$. Since $\Phi_g(x)\in U_g\cap V_g$, we have $\iota(V_g)=U_g$, that is, $\iota(\{V_g\}_{g\in G})=\{U_g\}_{g\in G}$.
	
	For (2), we have $\Or(\U)=\iota(\Or(\V))\subset\iota(\PO(\V))$ since $\Or(\V)\subset\PO(\V)$. Fix any $\{V_g\}_{g\in G}\in\PO(\V)$. Then there exists $(S,\V)$-pseudo-orbit $\{x_g\}_{g\in G}$ such that for any $s\in S$ and $g\in G$, $\Phi_s(x_g), x_{sg}\in V_{sg}\subset \iota(V_{sg})\in\U$, which implies that $\{x_g\}_{g\in G}$ is also $(S,\U)$-pseudo-orbit with pattern $\{\iota(V_g)\}_{g\in G}$. It is clear that $\iota(\PO(\V))\subset\PO(\U)$.
\end{proof}

By Theorem \ref{t:3.7}, we get the following corollary.

\begin{corollary}\label{c:4.2}
	Let $(X,G,\Phi)$ be a dynamical system, where $X$ is compact totally disconnected Hausdorff space. If $\{\U_\lambda\}_{\lambda\in\Lambda}$ is a cofinal directed subset of $\Part(X)$, then $(X,G,\Phi)$ is conjugate to $(\ilim\{\iota,\Or(\U_\lambda)\},G,\sigma^*)$ by the map $$\{u_\lambda\}_{\lambda\in\Lambda}\mapsto \bigcap_{\lambda\in\Lambda}\pi_{e_G}(u_\lambda).$$
\end{corollary}

Here, we use the same denotation $\iota$ to denote the map $\iota:\Or(\U_\lambda)\To \Or(\U_\eta)$ for each $\U_\lambda\succ\U_\eta$. Since for any $\U_\lambda\succ\U_\eta\succ\U_\gamma$, the composition of $\iota:\U_\lambda\To\U_\eta$ and $\iota:\U_\eta\To\U_\gamma$ equals to $\iota:\U_\lambda\To\U_\gamma$, the inverse limit of $(\iota, \Or(\U_\lambda))$ is well defined.
\begin{lemma}\label{l:3.4}
	Let $(X,G,\Phi)$ be a dynamical system, where $X$ is compact totally disconnected Hausdorff space. Then $\Phi$ has $S$-shadowing property if and only if for any $\U\in\Part(X)$ there exists $\V\in\Part(X)$ with $\V\succ\U$ such that for any $\W\in\Part(X)$ with $\W\succ\V$, $\iota(\PO(\W))=\Or(\U)$.
\end{lemma}

\begin{proof}
	First, suppose that $\Phi$ has $S$-shadowing property. Then for any $\U\in\Part(X)$, there exists $\V\in\FOC(X)$ which witnesses the $S$-shadowing property. Without loss of generality, we can take $\V\in\Part(X)$ and $\V\succ\U$. Fix any $(S,\V)$-pseudo-orbit $\{x_g\}_{g\in G}$ with pattern $\{V_g\}_{g\in G}$. By the $S$-shadowing property, there exists $z\in X$ $\U$-shadowing $\{x_g\}_{g\in G}$, that is, for any $g\in G$, $\Phi_g(z),x_g\in U_g$ for some $U_g\in\U$. Thus $x_g\in V_g\cap U_g$, which implies that $\iota(V_g)=U_g$. Noticing that $z\in\bigcap_{g\in G}\Phi_{g^{-1}}(U_g)$, it means that $\iota(\PO(\V))\subset\Or(\U)$.
	
	For the converse, fix any $\U\in\FOC(X)$. We take $\U'\in\Part(X)$ with $\U'\succ\U$. Then by the condition, there exists $\V\in\Part(X)$ with $\V\succ\U$ such that for any $\W\in\Part(X)$ with $\W\succ\V$, $\iota(\PO(\W))=\Or(\U')$. In particular, $\iota(\PO(\V))=\Or(\U')$. For any $(S,\V)$-pseudo-orbit $\{x_g\}_{g\in G}$ with pattern $\{V_g\}_{g\in G}$, we have  $\iota(\{V_g\}_{g\in G})\in\Or(\U')$. Take $z\in\bigcap_{g\in G}\Phi_{g^{-1}}(\iota(V_g))\nempty$. Thus $\Phi_g(z),x_g\in\iota(V_g)\in\U'$ for any $g\in G$, which implies that $z$ $\U'$-shadows $\{x_g\}_{g\in G}$. Also, $z$ $\U$-shadows $\{x_g\}_{g\in G}$ by $\U'\succ\U$.
\end{proof}

With the $S$-shadowing property, we can prove that the inverse limit of pseudo-orbit spaces is conjugate to the inverse limit of orbit spaces.
\begin{theorem}\label{t:4.4}
	Let $(X,G,\Phi)$ be a dynamical system, where $X$ is compact totally disconnected Hausdorff space, and $S$ be a finite subset of $G$. Let $\{\U_\lambda\}_{\lambda\in\Lambda}$ is a cofinal directed subset of $\Part(X)$. If $\Phi$ has $S$-shadowing property, then the system $(\ilim\{\iota,\PO(\U_\lambda)\},G,\sigma^*)$ is conjugate to $(\ilim\{\iota,\Or(\U_\lambda)\},G,\sigma^*)$.
	Also, those two systems satisfy the ML condition.
\end{theorem}

\begin{proof}
	Since $\Or(\U_\lambda)\subset\PO(\U_\lambda)$, the inclusion map $j:\Or(\U_\lambda)\mapsto\PO(\U_\lambda)$ induces the map $j^*:\ilim\{\iota,\Or(\U_\lambda)\}\mapsto\ilim\{\iota,\PO(\U_\lambda)\}$ by $j^*(\{u_\lambda\}_{\lambda\in\Lambda})=\{u_\lambda\}_{\lambda\in\Lambda}$, which is continuous injection and commute with $\sigma^*$. We will show that $j^*$ is surjection.
	
	By Lemma \ref{l:3.4}, we can take a monotone function $p:\Lambda\to\Lambda$ such that $\iota(\PO(\U_{p(\lambda)}))=\Or(\U_\lambda)$ and $\lambda\le p(\lambda)$, where $\lambda\le\eta$ if and only if $\U_\lambda\prec\U_\eta$. Thus we define $\psi:\ilim\{\iota,\PO(\U_\lambda)\}\mapsto\ilim\{\iota,\Or(\U_\lambda)\}$ by
	$$\psi(\{u_\lambda\}_{\lambda\in\Lambda})=\{\iota(u_{p(\lambda)})\}_{\lambda\in\Lambda}.$$
	
	It is well defined and continuous since $\iota(\PO(\U_{p(\lambda)}))=\Or(\U_\lambda)$, and commutes with $\sigma^*$.
	
	For each $\lambda\in\Lambda$ and $\{u_\gamma\}$, we have $(j^*\circ\psi(\{u_\gamma\}))_\lambda=(j^*(\{\iota(u_{p(\gamma)})\}))_\lambda=\iota(u_{p(\lambda)})=u_\lambda$. Thus $j^*\circ\psi$ is the identity map on $\ilim\{\iota,\PO(\U_\lambda)\}$, which implies that $j^*$ is surjection. It means that $(\ilim\{\iota,\PO(\U_\lambda)\},G,\sigma^*)$ is conjugate to $(\ilim\{\iota,\Or(\U_\lambda)\},G,\sigma^*)$.
	
	For the system $\PO(\U_\lambda)$, by Lemma \ref{l:3.4}, there exists $\eta\ge\lambda$ such that for any $\gamma\ge\eta$, $\iota(\PO(\U_\gamma))=\Or(\U_\lambda)=\iota(\PO(\U_\eta))$, which shows that $(\ilim\{\iota,\PO(\U_\lambda)\},G,\sigma^*)$ satisfies the ML condition. On the other hand, $(\ilim\{\iota,\Or(\U_\lambda)\},G,\sigma^*)$ satisfies the ML condition since $\iota: \Or(\U_\eta)\To\Or(\U_\lambda)$ is surjective for any $\eta\ge\lambda$.
\end{proof}

Combining Lemma \ref{l:3.6}, Corollary \ref{c:4.2} and Theorem \ref{t:4.4}, we get the following result.

\begin{corollary}\label{c:4.5}
	Suppose that $(X,G,\Phi)$ is a dynamical system, where $X$ is compact totally disconnected Hausdorff space, and $S$ be a finite subset of $G$. If $\Phi$ has $S$-shadowing property, then $(X,G,\Phi)$ is conjugate to the inverse limit of an ML inverse system of shifts of finite type over $S\cup \{e_G\}$.
\end{corollary}

\begin{proof}[Proof of Theorem \ref{t:1.1}]
	It follows immediately from Corollary \ref{c:4.5} and Theorem \ref{t:3.4}.
\end{proof}


\subsection{Shadowing in metric spaces}

In this subsection, we consider a dynamical system $(X,G,\Phi)$, where $X$ is compact metric systems.
For $\U\in\FOC(X)$ and $A\subset X$, the \emph{star} of $A$ in $\U$ is the set $st(A,\U)$ which is the union of all elements of $\U$ intersecting $A$. Denote by $\diam(A)$ the diameter of the set $A\subset X$. For a finite open cover $\U$, the diameter of $\U$ is $\diam(\U)=\max_{U\in \U}\diam(U)$, and the Lebesgue number of $\U$ is denoted by $\lambda(\U)$.


\begin{proof}[Proof of Theorem \ref{t:1.2}]
	First, we construct a sequence $\{\U_n\}$ of finite open covers such that
	\begin{itemize}
		\item[(1)] every $(S,\U_{n+1})$-pseudo-orbit can be $\U_n$-shadowed by some point,
		\item[(2)] $\{\U_n\}$ is a cofinal subset of $\FOC(X)$, and
		\item[(3)] for any $U\in\U_{n+2}$, there exists $W\in\U_n$ such that $st(U,\U_{n+1})\subset W$.
	\end{itemize}
Let $\U_0=\{X\}$.
Define $\U_{n+1}$ inductively by choosing a finite open cover $\U_{n+1}\succ\U_n$ which witness the $S$-shadowing property for $\U_n$, and $\diam(\U_{n+1})<\frac13\lambda(\U_n)$.
Since the definition of $\{\U_n\}$ and $\diam(\U_n)\To0$, $\{\U_n\}$ satisfies the condition (1) and (2).
Fix $n\in\N$ and $U\in\U_{n+2}$.
There exists $V\in\U_{n+1}$ such that $U\subset V$, which implies that $st(U,\U_{n+1})\subset st(V,\U_{n+1})$. Noticing that the diameter of $\diam(st(V,\U_{n+1}))\le 3\diam(\U_{n+1})<\lambda(\U_n)$, we can find $W\in\U_n$ such that
$W\supset st(V,\U_{n+1})\supset st(U,\U_{n+1})$. Thus $\{\U_n\}$ satisfies the condition (3).

For each $U\in\U_{n+2}$, fix $W(U)\in \U_n$ such that $W(U)\supset st(U,\U_{n+1})$, and define $\omega:\PO(\U_{n+2})\To\U^G$ by $\omega(\{U_g\}_{g\in G})=\{W(U_g)\}_{g\in G}$. It is clear that $\omega$ is well defined and continuous, and commutes with $\sigma$.

In fact, $\omega(\PO(\U_{n+2}))\subset \Or(\U_n)$. To see this, for any $\{U_g\}_{g\in G}\in\PO(\U_{n+2})$, there exists $(S,\U_{n+2})$-pseudo-orbit $\{x_g\}_{g\in G}$ with pattern $\{U_g\}_{g\in G}$. By the condition (1) of $\{\U_n\}$, there exists $z\in X$ which $\U_{n+1}$-shadows $\{x_g\}_{g\in G}$. Thus $\Phi_g(z),x_g\in st(U_g,\U_{n+1})\subset W(U_g)$ for any $g\in G$, which implies that $\{W(U_g)\}_{g\in G}\in\Or(\U_n)$.

Now, we define $\iota':\PO(\U_{n+2})\To\PO(\U_n)$ by $\iota'=j\circ\omega$, where $j$ is the inclusion map from $\Or(\U_n)$ to $\PO(\U_n)$. Also we define $\iota'':\Or(\U_{n+2})\To\Or(\U_n)$ by $\iota''=\omega\circ j'$, where $j'$ is the inclusion map from $\Or(\U_{n+2})$ to $\PO(\U_{n+2})$.
Thus $(\iota',(\PO(\U_{2i}),G,\sigma))$ and $(\iota'',(\Or(\U_{2i}),G,\sigma))$ form two inverse systems.
So the map $\omega$ induces a map $\omega^*$ from $\ilim\{\iota',\PO(\U_{2i})\}$ to $\ilim\{\iota'',\Or(\U_{2i})\}$ by $\omega^*(\{u_i\})=\{\omega(u_{i+1})\}$, which commutes with $\sigma^*$.
Noticing that $\Or(\U_{2(i+1)})\subset \PO(\U_{2(i+1)})$, we have for any $\{u_i\}\in\ilim\{\iota'',\Or(\U_{2i})\}$, $\omega(u_{i+1})=\omega\circ j'(u_{i+1})=\iota''(u_{i+1})=u_i$, which implies $\omega(\{u_i\})=\{u_i\}$ for any $\{u_i\}\in\ilim\{\iota'',\Or(\U_{2i})\}$.
Thus $\omega^*$ is surjective.

Finally, it ends proof by showing that $(X,G,\Phi)$ is a factor of $(\ilim\{\iota'',\Or(\U_{2i})\},G,\sigma)$.
Define $\psi:\ilim\{\iota'',\Or(\U_{2i})\}\To X$ by
$$\psi(\{u_i\})\in\bigcap_{i\in\N}\overline{\pi_{e_G}(u_i)},$$
and notice that $\pi_0(u_{i+1})\subset\pi_0(u_i)$ and $\diam(\U_{2i})\To 0$, which means that $\psi$ is well defined.
It is not hard to see that $\psi$ is continuous, and commutes with $\sigma^*$ from similar reasoning as Theorem \ref{t:3.7} and Corollary \ref{c:4.2}. To see $\psi$ is surjective, fix any $x\in X$. For any $i\in\N$, let $O_{2i}(x)\subset\Or(\U_{2i})$ be the set of all $\U_{2i}$-orbit pattern for $x$, and let $$O(x)=\bigcap_{i\in\N}\pi^{-1}_{2i}\left(\overline{O_{2i}(x)}\right)\cap\ilim\{\iota'',\Or(\U_{2i})\}.$$
So $O(x)$ is nonempty and $\psi(O(x))=\{x\}$, which implies that $\psi$ is surjective. Therefore, $(X,G,\Phi)$ is a factor of $(\ilim\{\iota',\PO(\U_{2i})\},G,\sigma)$ by a factor map $\psi\circ\omega^*$.
\end{proof}

\section{Shadowing of factor}

In this section, inspired by \cite{GM20}, we give a class of factor maps which preserve shadowing property.

\begin{definition}
	Let $(X,G,\Phi^X)$ and $(Y,G,\Phi^Y)$ be dynamical systems, where $S$ is a finite subset of $G$. We call the factor map $\phi:X\To Y$ \emph{lifts pseudo-orbit for $S$} if for any $\U_X\in\FOC(X)$, there exists $\U_Y\in\FOC(Y)$ such that for any $(S,\U_Y)$-pseudo-orbit $\{y_g\}_{g\in G}$, there exists a $(S,\U_X)$-pseudo-orbit $\{x_g\}_{g\in G}$ such that $\{y_g\}_{g\in G}=\{\phi(x_g)\}_{g\in G}$.
\end{definition}

\begin{theorem}
	Let $(X,G,\Phi^X)$ and $(Y,G,\Phi^Y)$  be dynamical systems, and $S$ be a finite subset of $G$. If $\Phi^X$ has $S$-shadowing property and $\phi:X\To Y$ is a factor map which lifts pseudo-orbit for $S$, then $\Phi^Y$ has $S$-shadowing property.
\end{theorem}

\begin{proof}
	Fix any $\U_Y\in\FOC(Y)$, and take $\U_X\in\FOC(X)$ such that $\U_X\succ\phi^{-1}\U_Y$. By the $S$-shadowing property of $\Phi^X$, there exists $\V_X\in\FOC(X)$ such that every $(S,\V_X)$-pseudo-orbit can be $\U_X$-shadowed by some point in $X$.
	Since $\phi$ lifts pseudo-orbit for $S$, let $\V_Y$ witness the property.
	For any $(S,\V_Y)$-pseudo-orbit $\{y_g\}_{g\in G}$, there exists $(S,\V_X)$-pseudo-orbit $\{x_g\}_{g\in G}$ such that $\phi(x_g)=y_g$ for any $g\in G$. Thus, there exists $z\in X$ which $\U_X$-shadows $\{x_g\}_{g\in G}$, that is, for any $g\in G$, there exists $U_g\in\U_X$ such that $\Phi^X_g(z),x_g\in U_g$. Since $\U_X\succ\phi^{-1}\U_Y$, it is clear that $\phi(z)$ $\U_Y$-shadows $\{y_g\}_{g\in G}$.
\end{proof}

\begin{definition}
	Let $(X,G,\Phi^X)$ and $(Y,G,\Phi^Y)$ be dynamical systems, and $S$ be a finite subset of $G$.
	For two sequence $\{x_g\}_{g\in G}$, $\{x'_g\}_{g\in G}$ of $X$ and $\U\in\FOC(X)$, we call $\{x_g\}_{g\in G}$ \emph{$\U$-shadows} $\{x'_g\}_{g\in G}$ if for any $g\in G$, there exists $U_g\in\U$ such that $x_g,x'_g\in U_g$.
	
	We call the factor map $\phi:X\To Y$ \emph{almost lifts pseudo-orbit for $S$} if for any $\U_X\in\FOC(X)$ and $\W_Y\in\FOC(Y)$, there exists $\U_Y\in\FOC(Y)$ such that for any $(S,\U_Y)$-pseudo-orbit $\{y_g\}_{g\in G}$, there exists $(S,\U_X)$-pseudo-orbit $\{x_g\}_{g\in G}$ such that $\{\phi(x_g)\}_{g\in G}$ $\W_Y$-shadows $\{y_g\}_{g\in G}$.
\end{definition}

The following theorem shows that a factor map that almost lifts pseudo-orbit preserves shadowing property.

\begin{theorem}\label{t:5.4}
	Suppose that $(X,G,\Phi^X)$ and $(Y,G,\Phi^Y)$ are dynamical systems, and $S$ be a finite subset of $G$. Let $\phi:X\To Y$ be a factor map. We have
	\begin{itemize}
		\item[(1)] if $\Phi^X$ has $S$-shadowing property and $\phi$ almost lifts pseudo-orbit for $S$, then $\Phi^Y$ has $S$-shadowing property, and
		\item[(2)] if $\Phi^Y$ has $S$-shadowing property, then $\phi$ almost lifts pseudo-orbit for $S$.
	\end{itemize}
\end{theorem}

\begin{proof}
	First, suppose that $\Phi^X$ has $S$-shadowing property and $\phi$ almost lifts pseudo-orbit for $S$.
	Fix any $\U_Y\in\FOC(Y)$.
	Then there exists $\W_Y$ such that for any two $W,W'\in\W_Y$ with $W\cap W'\nempty$, there exists $U\in\U_Y$ with $W\cup W'\subset U$.
	Take some $\U_X\in\FOC(X)$ with $\U_X\succ\phi^{-1}(\W_Y)$.
	For $\U_X$ and $\W_Y$, since $\phi$ almost lifts pseudo-orbit $S$, there exists $\V_Y$ witnessing the property.
	
	For any $(S,\V_Y)$-pseudo-orbit $\{y_g\}_{g\in G}$, there exists $(S,\U_X)$-pseudo-orbit $\{x_g\}_{g\in G}$ such that $\{\phi(x_g)\}_{g\in G}$ $\W_Y$-shadows $\{y_g\}_{g\in G}$, that is, for any $g\in G$ there exists $W\in\W_Y$ with $\phi(x_g),y_g\in W$.
	Since $\Phi^X$ has $S$-shadowing property, there is $z\in X$ which $\U_X$-shadows $\{x_g\}_{g\in G}$.
	So $\phi(z)$ $\W_Y$-shadows $\{\phi(x_g)\}_{g\in G}$ by noticing $\U_X\succ\phi^{-1}(\W_Y)$. So for any $g\in G$ there exists $W'\in\W_Y$ with $\Phi^Y_g(\phi(z)),\phi(x_g)\in W'$. For any $g\in G$, since $\phi(x_g)\in W\cap W'$, there exists $U_g\in U_Y$ such that $\Phi^Y_g(\phi(z)),\phi(x_g),y_g\in W\cup W'\subset U_g$.
	Thus $\phi(z)$ $\U_Y$-shadows $\{y_g\}_{g\in G}$.
	
	To prove $(2)$, fix any $\U_X\in\FOC(X)$ and $\W_Y\in\FOC(Y)$.
	Since $\Phi^Y$ has $S$-shadowing property, let $\V_Y$ witness the property with respect ot $\W_Y$.
	For any $(S,\V_Y)$-pseudo-orbit $\{y_g\}_{g\in G}$, there exists $z\in Y$ which $\W_Y$-shadows $\{y_g\}_{g\in G}$. Take some $x\in\phi^{-1}(z)$. So $\{\Phi^X_g(x)\}_{g\in G}$ is $(S,\U_X)$-pseudo-orbit, and $\{\phi(\Phi^X_g(x))\}_{g\in G}=\{\Phi^Y_g(z)\}$ $\W_Y$-shadows $\{y_g\}_{g\in G}$. Thus $\phi$ almost lifts pseudo-orbit for $S$.
\end{proof}

For a metric space, we have the following result.

\begin{lemma}\label{l:5.5}
	Suppose that $(X,G,\Phi^X)$ and $(Y,G,\Phi^Y)$ are dynamical systems, where $X,Y$ are compact metric spaces and $S$ is a finite subset of $G$. Let $\phi:X\To Y$ be a factor map. Then $\phi$ almost lifts pseudo-orbit for $S$ if and only if for any $\eps,\eta>0$, there exists $\delta>0$ such that for any $(S,\delta)$-pseudo-orbit $\{y_g\}_{g\in G}$ of $Y$, there exists $(S,\eta)$-pseudo-orbit $\{x_g\}_{g\in G}$ of $X$ such that $d_Y(\phi(x_g),y_g)<\eps$ for any $g\in G$.
\end{lemma}

\begin{proof}
	First, suppose that $\phi$ almost lifts pseudo-orbit for $S$.
	Fix any $\eps,\eta>0$.
	Take $\U_X\in\FOC(X)$ with $\diam(\U_X)<\eta$ and $\W_Y\in\FOC(Y)$ with $\diam(\W_Y)<\eps$.
	Since $\phi$ almost lifts pseudo-orbit for $S$, there exists $\U_Y$ witnessing the property with respect to $\U_X,\W_Y$.
	Take $\delta=\lambda(\U_Y)$.
	Fix any $(S,\delta)$-pseudo-orbit $\{y_g\}_{g\in G}$ of $Y$.
	For each $g\in G$, choose $U_g\in\U_Y$ such that $B(y_g,\delta)\subset U_g$.
	Then for any $g\in G$ and $s\in S$, $d_X(y_{sg},\Phi^X_s(y_g))<\delta$, which implies that $y_{sg},\Phi^X_s(y_g)\in U_{sg}$.
	So $\{y_g\}_{g\in G}$ is $(S,\U_Y)$-pseudo-orbit with pattern $\{U_g\}_{g\in G}$.
	Then there exists $(S,\U_X)$-pseudo-orbit $\{x_g\}_{g\in G}$ such that $\{\phi(x_g)\}_{g\in G}$ $\W_Y$-shadows $\{y_g\}_{g\in G}$. Since $\diam(\U_X)<\eta$ and $\diam(\W_Y)<\eps$, $\{x_g\}_{g\in G}$ is $(S,\eta)$-pseudo-orbit and $d_Y(\phi(x_g),y_g)<\eps$ for any $g\in G$.
	
	For the converse, fix any $\U_X\in\FOC(X)$ and $\W_Y\in\FOC(Y)$.
	Take $\eps=\lambda(\W_Y)$ and $\eta=\lambda(\U_X)$, and let $\delta$ witness the condition with respect to $\eps,\eta$.
	Choose $\V_Y\in\FOC(X)$ with $\diam(\V_Y)<\delta$.
	For any $(S,\V_Y)$-pseudo-orbit $\{y_g\}_{g\in G}$, it is also $(S,\delta)$-pseudo-orbit of $Y$.
	Thus there exists $(S,\eta)$-pseudo-orbit $\{x_g\}_{g\in G}$ of $X$ such that $d_Y(\phi(x_g),y_g)<\eps$ for any $g\in G$.
	For each $g\in G$, fix $U_g\in\U_X$ such that $B(x_g,\eta)\subset U_g$. Thus for any $g\in G$ and $s\in S$, $d_X(y_{sg},\Phi^X_s(y_g))<\eta$, which implies that $y_{sg},\Phi^X_s(y_g)\in U_{sg}$. So $\{x_g\}_{g\in G}$ is $(S,\U_X)$-pseudo-orbit with pattern $\{U_g\}_{g\in G}$. Since $d_Y(\phi(x_g),y_g)<\eps$ for any $g\in G$, it is clear that $\{\phi(x_g)\}_{g\in G}$ $\W_Y$-shadows $\{y_g\}_{g\in G}$.
\end{proof}

By Theorem \ref{t:3.4} and Theorem \ref{t:5.4}, we have the following corollary.

\begin{corollary}
	Let $(X,G,\Phi)$ be a dynamical system, where $X$ is a compact metric space, and $S$ be a finite subset of $G$. If $(X,G,\Phi)$ is a factor of an ML inverse limit of a sequence of subshifts which have $S$-shadowing property by a map which almost lifts pseudo-orbit for $S$, then $\Phi$ has $S$-shadowing property.
\end{corollary}

\section{Some dynamical properties preserving by inverse limit}

In this section, we investigated other dynamical properties $P$ satisfying that systems with property $P$ can be represented by subshifts with the same property.

Let $(X,G,\Phi)$ be a dynamical system, and for any two nonempty sets $U,V\subset X$, denote the \emph{hitting time set} from $U$ to $V$ by
$$N(U,V)=\{g\in G\setminus\{e_G\}:U\cap \Phi_{g^{-1}}(V)\nempty\},$$
and $N(x,U):=N(\{x\},U)$ for any $x\in X$.

Recall that $(X,G,\Phi)$ is \emph{transitive} if for any two nonempty open sets $U,V\subset X$, $N(U,V)$ is infinite.
The dynamical system $(X,G,\Phi)$ is called \emph{totally transitive} if for any countable subset $K\subset G$ and two nonempty open sets $U,V\subset X$, $N(U,V)\cap G'\nempty$ where $G'$ is the group generated by $K$.
The dynamical system $(X,G,\Phi)$ is called \emph{minimal} if it does not contain any proper subsystem. It is proved that dynamical system $(X,G,\Phi)$ is minimal if for any $x\in X$ and any nonempty open set $U\subset X$, $N(x,U)$ is syndetic, that is, there exists finite subset $F\subset G$, $FN(x,V)=G$.
The dynamical system $(X,G,\Phi)$ is called \emph{weakly mixing} if $(X\times X,G,\Phi\times \Phi)$ is transitive.\footnote{Naturally, $\Phi\times\Phi$ is defined as $(\Phi\times\Phi)_g=\Phi_g\times \Phi_g$.}
The dynamical system $(X,G,\Phi)$ is called \emph{mixing} if for any two nonempty open sets $U,V\subset X$, $G\setminus N(U,V)$ is finite.

The following lemmas show the connection of hitting time set between the inverse limit and the inverse system.

\begin{lemma}\label{l:6.1}
	Let $(\Lambda,\le)$ be a directed set. Suppose that  $(X,G,\Phi^*)$ be an inverse limit of an ML inverse system $(\phi^\eta_\lambda,(X_\lambda,G,\Phi^\lambda))$. Then for any two nonempty open sets $U,V$ of $X$, there exists $\gamma\in\Lambda$ such that for any
	$\eta\ge\gamma$, there exist two nonempty open sets $U_\eta,V_\eta\subset X_\eta$ such that $N(U_\eta,V_\eta)\subset N(U,V)$.
\end{lemma}

\begin{proof}
	For any two nonempty open sets $U,V$ of $X$, there exist $\lambda\in\Lambda$ and two nonempty open sets $U_\lambda,V_\lambda$ such that $\emptyset\neq\pi_\lambda^{-1}(U_\lambda)\cap X\subset U$ and $\emptyset\neq\pi_\lambda^{-1}(V_\lambda)\cap X\subset V$.
	By the ML condition, there exists $\gamma\ge\lambda$ such that for any $\eta\ge\gamma$, $\phi^\eta_\lambda(X_\eta)=\phi^\gamma_\lambda(X_\gamma)$.
	Let $U_\eta=(\phi^\eta_\lambda)^{-1}(U_\lambda)$ and  $V_\eta=(\phi^\eta_\lambda)^{-1}(V_\lambda)$.
	Since $\pi_\lambda^{-1}(U_\lambda)\cap X\nempty$, $U_\eta$ is not empty set, and it is also same for $V_\eta$.
	
	To end the proof, it remains to prove $N(U_\eta,V_\eta)\subset N(U,V)$.
	Fix any $g\in N(U_\eta,V_\eta)$.
	Then there exists $x_\eta\in U_\eta$ such that $\Phi^\eta_g(x_\eta)\in V_\eta$.
	Since $\phi^\eta_\lambda(x_\eta)\in\phi^\eta_\lambda(X_\eta)$, there exists $x\in X$ such that $\pi_\lambda(x)=\phi^\eta_\lambda(x_\eta)$.
	It is clear that $x\in\pi_\lambda^{-1}(U_\lambda)\cap X\subset U$.
	And $\pi_\lambda(\Phi^*_g(x))=\Phi^\lambda_g(\phi^\eta_\lambda(x_\eta))=\phi^\eta_\lambda(\Phi^\eta_g(x_\eta))\in V_\lambda$, which implies that $\Phi^*_g(x)\in V$.
	It implies that $N(U_\eta,V_\eta)\subset N(U,V)$.
\end{proof}

\begin{lemma}\label{l:6.2}
	Let $(\Lambda,\le)$ be a directed set. Suppose that  $(X,G,\Phi^*)$ be an inverse limit of an ML inverse system $(\phi^\eta_\lambda,(X_\lambda,G,\Phi^\lambda))$. Then for any $x\in X$ and any nonempty open set $U$ of $X$, there exist $\lambda\in\Lambda$, and nonempty open set $U_\lambda\subset X_\lambda$ such that $N(\pi_\lambda(x),U_\lambda)\subset N(x,U)$.
\end{lemma}

\begin{proof}
	Fix any $x\in X$ and any nonempty open set $U$ of $X$. Then there exist $\lambda\in\Lambda$ and two nonempty open set $U_\lambda\subset X_\lambda$ such that $\emptyset\neq\pi_\lambda^{-1}(U_\lambda)\cap X\subset U$.
	
	Fix any $h\in N(\pi_\lambda(x),U_\lambda)$. Then
	$$\pi_\lambda(\Phi^*_h(x))=\Phi^\lambda_h(\pi_\lambda(x))\in U_\lambda,$$
	which implies that $\Phi^*_h(x)\in\pi_\lambda^{-1}(U_\lambda)\cap X\subset U$ and ends the proof.
\end{proof}

\subsection{Proof of Theorem \ref{t:1.3}}
By Theorem \ref{t:3.7}, it is shown that a dynamical system can be represented by an inverse limit of its orbit space.
So we consider the connection of hitting time set between the dynamical system and its orbit space.

\begin{lemma}\label{l:6.3}
	Let $(X,G,\Phi)$ be a dynamical system with compact Hausdorff space $X$, $\U\in\FOC(X)$, and $A$ be a finite subset of $G$. Then for any $P_A,P'_A\in\mathcal{L}_A(\Or(\U))$, we have $N(U,V)\setminus (A^{-1}A)\subset N(C(P_A)\cap \Or(\U),C(P'_A)\cap \Or(\U))\subset N(U,V)$ where
	$$U=\bigcap_{g\in A}\Phi_{g^{-1}}(P_A(g)),V=\bigcap_{g\in A}\Phi_{g^{-1}}(P'_A(g)).$$
\end{lemma}

\begin{proof}
	Fix any $P_A,P'_A\in\mathcal{L}_A(\Or(\U))$.
	And let
	$$U=\bigcap_{g\in A}\Phi_{g^{-1}}(P_A(g)),V=\bigcap_{g\in A}\Phi_{g^{-1}}(P'_A(g)).$$
	
	First, fix any $h\in N(U,V)\setminus (A^{-1}A)$.
	Then there exists $x\in X$ such that $x\in U$ and $\Phi_h(x)\in V$.
	Let $B=A\cup Ah$, and $u\in\U^G$ satisfy $$\pi_g(u)=P_A(g),\,\,\pi_{gh}(u)=P'_A(g),\,g\in A,$$
	and
	$$\Phi_g(x)\in\pi_g(u),\,g\in G\setminus B.$$
	Since $h\notin A^{-1}A$, we have $A\cap Ah=\emptyset$, which implies that $u$ is well defined.
	By the definition of $U$ and $V$, $x\in U$ and $\Phi_h(x)\in V$ imply that $x\in\bigcap_{g\in G}\Phi_{g^{-1}}(\pi_g(u))$. Thus $u\in\Or(\U)$, $u\in C(P_A)\cap \Or(\U)$ and $\sigma_h(u)\in C(P'_A)\cap \Or(\U)$. Therefore, $h\in N(C(P_A)\cap \Or(\U),C(P'_A)\cap \Or(\U))$.
	
	Finally, for any $h\in N(C(P_A)\cap \Or(\U),C(P'_A)\cap \Or(\U))$, there exists $u\in\Or(\U)$ such that $u\in C(P_A)\cap \Or(\U)$ and $\sigma_h(u)\in C(P'_A)\cap \Or(\U)$.
	Let $B=A\cup Ah$.
	Since $u\in\Or(\U)$, take $x\in\bigcap_{g\in B}\Phi_{g^{-1}}(\pi_g(u))$, noticing that for any finite subset $K\subset G$,
	$$\mathcal{L}_K(\Or(\U))=\{P_K\in\U^K:\bigcap_{g\in K}\Phi_{g^{-1}}P_K(g)\nempty\}.$$
	In particular, $x\in\bigcap_{g\in A}\Phi_{g^{-1}}(\pi_g(u))$, and by $\pi_g(u)=P_A(g)$ for any $g\in A$.
	Thus $x\in U$.
	For any $g\in A$, $\Phi_{gh}(x)\in\pi_{gh}(u)=\pi_g(\sigma_h(u))$, that is $\Phi_h(x)\in\bigcap_{g\in A}\Phi_{g^{-1}}(\pi_g(\sigma_h(u)))$.
	So $\Phi_h(x)\in V$.
	Thus $h\in N(U,V)$.
\end{proof}

\begin{remark}
	If $X$ is totally disconnected space and $\U\in\Part(X)$, it can be similarly proved that $N(C(P_A)\cap \Or(\U),C(P'_A)\cap \Or(\U))=N(U,V)$, since for some $a,a'\in A$ with $a=a'h$, we have $\Phi_a(x)=\Phi_{a'h}(x)\in P_A(a)\cap P'_A(a')$ and then $P_A(a)=P'_A(a')$.
\end{remark}

\begin{lemma}\label{l:6.5}
	Let $(X,G,\Phi)$ be a dynamical system with compact Hausdorff space $X$, $\U\in\Part(X)$, and $A$ be a finite subset of $G$. Then for any $u\in\Or(\U)$ and any $P_A\in\mathcal{L}_A(\Or(\U))$, we have $N(u,C(P_A)\cap \Or(\U))=N(x,U)$ where
	$$x\in\bigcap_{g\in G}\Phi_{g^{-1}}(\pi_g(u)),U=\bigcap_{g\in A}\Phi_{g^{-1}}(P_A(g)).$$
\end{lemma}

\begin{proof}
	First, fix any $h\in N(x,U)$.
	Thus for any $g\in A$, we have $\Phi_h(x)\in \Phi_{g^{-1}}(P_A(g))$.
	Since $\Phi_{gh}(y)\in\pi_{gh}(u)$, we have $\Phi_h(y)\in\pi_{gh}(u)\cap P_A(g)$, that is, $\pi_g(\sigma_h(u))=\pi_{gh}(u)=P_A(g)$. Therefore, $\sigma_h(u)\in C(P_A)\cap \Or(\U)$.
	
	For the converse, fix any $h\in N(u,C(P_A)\cap \Or(\U))$ and any $x\in\bigcap_{g\in G}\Phi_{g^{-1}}(\pi_g(u))$.
	Thus, for any $g\in A$, $\Phi_{gh}(x)\in\pi_{gh}(u)=\pi_g(\sigma_h(u))=P_A(g)$. Then $\Phi_h(x)\in U$.	
\end{proof}

Notice that $\ilim\{\phi^\eta_\lambda,X_\lambda\}\times\ilim\{\phi^\eta_\lambda,X_\lambda\}$ is conjugate to $(\ilim\{\phi^\eta_\lambda\times\phi^\eta_\lambda,X_\lambda\times X_\lambda\}$ by the map
$$(\{x_\lambda\}_{\lambda\in\Lambda},\{x'_\lambda\}_{\lambda\in\Lambda})\mapsto \{(x_\lambda,x'_\lambda)\}_{\lambda\in\Lambda},$$
and $\Or(\U)\times\Or(\U)$ is conjugate to $\Or(\U\times\U)$ by the map
$$(\{U_g\}_{g\in G},\{V_g\}_{g\in G})\mapsto \{U_g\times V_g\}_{g\in G},$$
where $\U\times\U:=\{U\times V:U,V\in\U\}$.
Thus by Lemma \ref{l:6.1}, Lemma \ref{l:6.2}, Lemma \ref{l:6.3} and Lemma \ref{l:6.5}, we get the following lemmas.
\begin{lemma}\label{l:6.6}
	Let $(\Lambda,\le)$ be a directed set. Suppose that  $(X,G,\Phi^*)$ be an inverse limit of an ML inverse system $(\phi^\eta_\lambda,(X_\lambda,G,\Phi^\lambda))$.
	If $(X_\lambda,G,\Phi^\lambda)$ is transitive (minimal, totally transitive, weakly mixing, or mixing) for each $\lambda\in\Lambda$, then $(X,G,\Phi^*)$ is transitive (minimal, totally transitive, weakly mixing, or mixing).
\end{lemma}

\begin{lemma}\label{l:6.7}
	Suppose that $(X,G,\Phi)$ is a dynamical system and $\U\in\FOC(X)$.
	If $(X,G,\Phi)$ is transitive (totally transitive, weakly mixing, or mixing), then $(\Or(\U),G,\sigma)$ is transitive (totally transitive, weakly mixing, or mixing). Moreover, if $\U\in\Part(X)$ and $(X,G,\Phi)$ is minimal, then $(\Or(\U),G,\sigma)$ is minimal.
\end{lemma}

We also prove a similar result for specification property in totally disconnected space. We used the definition of specification in \cite{CL15}. For a dynamical system $(X,G,\Phi)$ with compact metric space $X$, $\Phi$ has \emph{specification property} if for any $\eps>0$, there exists nonempty finite subset $F\subset G$ with the following property: for any finite collection $F_1,F_2,...,F_m$ of finite subsets of $G$ with
$$FF_i\cap F_j=\emptyset\text{ for any } 1\le i,j\le m\text{ and } i\neq j, $$
and for any $x_1,x_2,...,x_m\in X$, there exists $y\in X$ such that $d(\Phi_s(y),\Phi_s(x_i))<\eps$ for each $s\in F_i$ and $i=1,2,...,m$.

The following lemma shows the topological definition of specification property.

\begin{lemma}
	 For a dynamical system $(X,G,\Phi)$ with compact metric space $X$, $\Phi$ has specification property if and only if for any $\U\in\FOC(X)$, there exists nonempty finite subset $F\subset G$ with the following property: for any finite collection $F_1,F_2,...,F_m$ of finite subsets of $G$ with
	 $$FF_i\cap F_j=\emptyset\text{ for any } 1\le i,j\le m\text{ and } i\neq j, $$
	 and for any $x_1,x_2,...,x_m\in X$, there exists $y\in X$ such that for each $i=1,2,...,m$ and each $s\in F_i$, there exists $U_{(i,s)}\in\U$ with $\Phi_s(y),\Phi_s(x_i)\in U_{(i,s)}$.
\end{lemma}

\begin{proof}
	First, suppose that $\Phi$ has specification property.
	For any $\U\in\FOC(X)$, let $\eps$ be the Lebesgue number of $\U$, and finite subset $F\subset G$ witness the specification property with respect to $\eps$.
	Fix any finite collection $F_1,F_2,...,F_m$ of finite subsets of $G$ with
	$$FF_i\cap F_j=\emptyset\text{ for any } 1\le i,j\le m\text{ and } i\neq j, $$
	and for any $x_1,x_2,...,x_m\in X$, there exists $y\in X$ such that for each $i=1,2,...,m$ and each $s\in F_i$, $d(\Phi_s(y),\Phi_s(x_i))<\eps$, which implies that there exists $U_{(i,s)}\in\U$, $\Phi_s(y),\Phi_s(x_i)\in U_{(i,s)}$.
	
	For the converse, fix any $\eps>0$.
	Let $\U$ be a finite subcover of $\{B(x,\eps/2):x\in X\}$, and finite subset $F\subset G$ witness the condition with respect to $\U$.
	Similarly, by $\diam(\U)<\eps$, we can get $\Phi$ has specification property.
\end{proof}

So the definition of specification property can be also generalized to compact Hausdorff space.
For a dynamical system $(X,G,\Phi)$ with compact Hausdorff space $X$, $\Phi$ (or $(X,G,\Phi)$) has \emph{specification property} if, for any $\U\in\FOC(X)$, there exists nonempty finite subset $F\subset G$ with the following property: for any finite collection $F_1,F_2,...,F_m$ of finite subsets of $G$ with
$$FF_i\cap F_j=\emptyset\text{ for any } 1\le i,j\le m\text{ and } i\neq j, $$
and for any $x_1,x_2,...,x_m\in X$, there exists $y\in X$ such that for each $i=1,2,...,m$ and each $s\in F_i$, there exists $U_{(i,s)}\in\U$ with $\Phi_s(y),\Phi_s(x_i)\in U_{(i,s)}$.

\begin{lemma}\label{l:6.9}
	Let $(\Lambda,\le)$ be a directed set. Suppose that  $(X,G,\Phi^*)$ be an inverse limit of an ML inverse system $(\phi^\eta_\lambda,(X_\lambda,G,\Phi^\lambda))$.
	If there exists $\eta\in\Lambda$ such that $\Phi^\lambda$ has specification property for each $\lambda\ge\eta$, then $\Phi^*$ has specification property.
\end{lemma}

\begin{proof}
	Fix any $\U\in\FOC(X)$.
	There exist $\lambda$ and $\U_\lambda\in\FOC(X_\lambda)$ such that $\{(\pi_\lambda)^{-1}(U_\lambda)\cap X:U_\lambda\in\U_\lambda\}\succ \U$.
	Take $\eta'$ satisfying $\eta'\ge\eta$ and $\eta'\ge\lambda$.
	By the ML condition, there exists $\gamma\ge\eta'$ such that for any $\mu\ge\gamma$, $\phi^\mu_{\eta'}(X_\mu)=\phi^\gamma_{\eta'}(X_\gamma)$.
	Let $\U_\gamma=\{(\phi^\gamma_{\lambda})^{-1}(U_\lambda):U_\lambda\in\U_\lambda\}$.
	Since $X_\gamma$ has specification property, there exists a finite subset $F\subset G$ which witnesses the specification property with respect to $\U_\gamma$.
	
	Fix a finite collection $F_1,F_2,...,F_m$ of finite subsets of $G$ with
	$$FF_i\cap F_j=\emptyset\text{ for any } 1\le i,j\le m\text{ and } i\neq j, $$
	and $m$ points $x^1,x^2,...,x^m\in X$.
	For $x_\gamma^i=\pi_\gamma(x^i)$, $i=1,2,...,m$, by the specification property of $\Phi^\gamma$, there exists $y_\gamma\in X_\gamma$ such that for each $i=1,2,...,m$ and each $s\in F_i$, there exists $U_{(i,s,\gamma)}\in\U_\gamma$, $\Phi^\gamma_s(y_\gamma),\Phi^\gamma_s(x_\gamma^i)\in U_{(i,s,\gamma)}$.
	Then there exists $y\in X$ with $\pi_\lambda(y)=\phi^\gamma_{\lambda}(y_\gamma)$.
	For each $i=1,2,...,m$ and each $s\in F_i$, let $U_{(i,s)}\in\U_\lambda$ satisfy $U_{(i,s,\gamma)}=(\phi^\gamma_{\lambda})^{-1}(U_{(i,s)})$.
	Thus for each $i=1,2,...,m$ and each $s\in F_i$,
	$$\pi_\lambda(\Phi^*_s(y))=\Phi_s^\lambda(\pi_\lambda(y))=\Phi_s^\lambda(\phi^\gamma_{\lambda}(y_\gamma))=\phi^\gamma_{\lambda}(\Phi_s^\gamma(y_\gamma))\in U_{(i,s)},$$
	and
	$$\pi_\lambda(\Phi^*_s(x^i))=\Phi_s^\lambda(\pi_\lambda(x^i))=\Phi_s^\lambda(\phi^\gamma_{\lambda}(x^i_\gamma))=\phi^\gamma_{\lambda}(\Phi_s^\gamma(x^i_\gamma))\in U_{(i,s)},$$
	which ends the proof by $\{(\pi_\lambda)^{-1}(U_\lambda)\cap X:U_\lambda\in\U_\lambda\}\succ \U$.
\end{proof}

On the other hand, we consider the relationship of specification property between the system and its orbit space.

\begin{lemma}\label{l:6.10}
	Let $(X,G,\Phi)$ be a dynamical system with compact totally disconnected space $X$, and $\U\in\Part(X)$. If $\Phi$ has specification property, then $(\Or(\U),G,\sigma)$ has specification property.
\end{lemma}

\begin{proof}
	Without loss of generality, fix any finite subset $A\subset G$.
	So $\{C(P_A)\cap \Or(\U):P_A\in\mathcal{L}_A(\Or(\U))\}\in\Part(\Or(\U))$.
	Let
	$$\V=\{\bigcap_{g\in A}\Phi_{g^{-1}}(P_A(g)):P_A\in\mathcal{L}_A(\Or(\U))\}.$$
	Thus $\V\in \Part(X)$.
	Since $\Phi$ has specification property, there exists a finite subset $F\subset G$ which witnesses the property with respect to $\V$.
	
	Fix a finite collection $F_1,F_2,...,F_m$ of finite subsets of $G$ with
	$$FF_i\cap F_j=\emptyset\text{ for any } 1\le i,j\le m\text{ and } i\neq j, $$
	and $u_1,u_2,...,u_m\in\Or(\U)$.
	Take $x_i\in\bigcap_{g\in G}\Phi_{g^{-1}}(\pi_g(u_i))$.
	Then there exists $y\in X$ satisfying that, for each $i=1,2,...,m$ and each $s\in F_i$, there exists $V_{i,s}\in\V$ such that $\Phi_s(y),\Phi_s(x_i)\in V_{i,s}$.
	Let $P_A^{(i,s)}\in\mathcal{L}_A(\Or(\U))$ satisfy $V_{i,s}=\bigcap_{g\in A}\Phi_{g^{-1}}(P^{(i,s)}_A(g))$ and $w\in\Or(\U)$ satisfy $\Phi_g(y)\in\pi_g(w)$.
	Thus $\Phi_s(y),\Phi_s(x_i)\in V_{i,s}$ implies that $\Phi_{gs}(y),\Phi_{gs}(x_i)\in P^{(i,s)}_A(g)$.
	So $\pi_{gs}(w)=P^{(i,s)}_A(g)=\pi_{gs}(u_i)$ for each $g\in A$.
	Therefore, for each $i=1,2,...,m$ and each $s\in F_i$, we have $\sigma_s(w),\sigma_s(u_i)\in C(P^{(i,s)}_A)$, which ends the proof.
\end{proof}

\begin{remark}
	If $X$ is not totally disconnected, we do not know whether the above lemma holds.
	Although $\Phi_{gs}(x_i)\in P^{(i,s)}_A(g)\cap\pi_{gs}(u_i)$ and we can choose $\omega$ such that $\pi_{gs}(\omega)=P^{(i,s)}_A(g)$, it may happen that $\Phi_{gs}(y)\in P^{(i,s)}_A(g)\setminus \pi_{gs}(u_i)$. In this case, $\omega$ can not be the tracing point.
\end{remark}


Now, we give the proof of Theorem \ref{t:1.3}.

\begin{proof}[Proof of Theorem \ref{t:1.3}]
	Let $(X,G,\Phi)$ be a dynamical system with compact totally disconnected space $X$. By Corollary \ref{c:4.2}, $(X,G,\Phi)$ is conjugate to $(\ilim\{\iota,\Or(\U_\lambda)\},G,\sigma^*)$. Let property $P$ be one of the following properties: transitivity, minimal, totally transitivity, weakly mixing, mixing, and specification property. By Lemma \ref{l:6.6}, Lemma \ref{l:6.7}, Lemma \ref{l:6.9} and Lemma \ref{l:6.10}, it ends the proof.
\end{proof}

\begin{remark}
	More generally, we can use the notion of Furstenberg family (See details in \cite{Aki97} and \cite{Fur81}). Let $2^G$ be the power set of $G$. A collection  $\mathcal{F}\subset 2^G$ is called \emph{Furstenberg family} if $\mathcal{F}$ is upper hereditary, that is, $F_1\subset F_2$ and $F_1\in\mathcal{F}$ imply that $F_2\in\mathcal{F}$.
	For a Furstenberg family $\mathcal{F}$, the dynamical system $(X,G,\Phi)$ is called \emph{$\mathcal{F}$-transitive} if for any two nonempty open sets $U,V\subset X$, $N(U,V)\in\mathcal{F}$.
	Let
	$$\mathcal{F}_{\mathrm{inf}}=\{F\subset G:\#F=\infty\},$$
	$$\mathcal{F}_{\mathrm{t}}=\{F\subset G:\text{for any finite subset $S\subset G$, there exists $g\in G$ such that $Sg\subset F$}\},$$
	$$\mathcal{F}_{\mathrm{cf}}=\{F\subset G:\#(G\setminus F)<\infty\},$$
	
	It is shown that transitive, weakly mixing and mixing is equivalent to $\mathcal{F}_{\mathrm{inf}}$-transitive
	, $\mathcal{F}_{\mathrm{t}}$-transitive (if $G$ is abelian)
	and $\mathcal{F}_{\mathrm{cf}}$-transitive, respectively.
	
	If a Furstenberg family $\mathcal{F}\subset \mathcal{F}_{\mathrm{inf}}$ satisfying that $F\setminus A\in\mathcal{F}$ for any $F\in\mathcal{F}$ and any finite subset $A\subset G$, then Theorem \ref{t:1.3} holds when property $P$ is $\mathcal{F}$-transitive.
\end{remark}

In the case of $\N$, that is, a dynamical system $(X,f)$, where $X$ is compact Hausdorff space and $f:X\To X$ is a continuous map. We can also define the hitting time set $N(U,V):=\{n\in\N^*:U\cap f^{-n}(V)\nempty\}$ for two nonempty sets $U,V$ of $X$. We also have similar results of Lemma \ref{l:6.6}, Lemma \ref{l:6.7}, Lemma \ref{l:6.9} and Lemma \ref{l:6.10}. And by \cite[Corollary 14]{GM20},  we have the following corollary.

\begin{corollary}\label{c:6.13}
	Suppose that $f:X\To X$ is a continuous surjection, where $X$ is compact totally disconnected Hausdorff space. Let property $P$ be one of the following properties: transitivity, minimal, totally transitivity, weakly mixing, mixing and specification property. Then $f$ has property $P$ if and only if it is conjugate to the inverse limit of an ML inverse system of subshifts with property $P$.
\end{corollary}

\subsection{Proof of Theorem \ref{t:1.4}}
For the metric case, the natural map $\iota$ defined in the case of totally disconnected space can not deduce to the metric case, since there may be two images satisfying $\iota(U)\subset U$.
If we fix $\iota(U)$ for each $U\in\U_n$, the inverse system $(\iota',(\Or(\U_n),G,\sigma))$ where the bonding map $\iota':\Or(\U_{n+1})\To\Or(\U_n)$ is induced by $\iota$ may fail to ensure the ML condition.
Hence we constructed another more suitable subshift which is also induced by finite open covers.

First, we constructed a sequence $\{\V_n\}\subset \FOC(X)$ of finite open covers satisfying that
\begin{itemize}
	\item[(1)] for any $n\in\N$ and any $U\in\V_n$, $\bigcup\{V\in\V_{n+1}:V\subset U\}=U$,
	\item[(2)] for any $n\in\N$, $\V_{n+1}\succ\V_n$, and
	\item[(3)] $\lim_{n\To\infty}\diam(\V_n)=0$.
\end{itemize}
For example, take $\V_0$ be a finite open subcover of $\{B(x,1):x\in X\}$.
Suppose that $\V_n$ is defined, choose a finite open subcover $\V'_{n+1}$ of $\{B(x,2^{-n-1}):x\in X\}$ and take $\V_{n+1}=\V_n\vee \V'_{n+1}$, where $\U\vee\V:=\{U\cap V\nempty:U\in\U,V\in\V\}$. Thus $\{\V_n\}$ satisfies the above conditions.

For each $n,k\in\N$, define
$$\U_n=\{(V_0,V_1,...,V_n)\in\V_1\times\V_2\times\cdots\times\V_n: V_0\supset V_1 \supset\cdots\supset V_n\},$$
and $\iota^{n+k}_n:\U_{n+k}\To\U_n$ by
$$\iota^{n+k}_n((V_0,V_1,...,V_{n+k}))=(V_0,V_1,V_2,...,V_n).$$
It is clear that $\iota^{n+k}_n$ is surjective.

Similar to the case of totally disconnected space, $\U_n$ can also induce a subshift.
For convenience, we defined $\kappa:\U_n\To\V_n$ by $\kappa((V_0,V_1,...,V_n))=V_n$, that is, the projection onto the last coordinate.
Here, we do not distinguish between $\kappa$ for different $n\in\N$.
Define $\Or'(\U_n)$ be the closure of
$$\Or''(\U_n):=\{\{(V^g_0,V^g_1,...,V^g_n)\}_{g\in G}\in\U_n^G:\bigcap_{g\in G}\Phi_{g^{-1}}(V^g_n)\nempty\}.$$
Similar to the orbit space, $\Or'(\U_n)$ is also a subshift of $\U_n^G$, and
$$\mathcal{L}_A(\Or'(\U_n))=\{P_A\in\U_n^A:\bigcap_{g\in A}\Phi_{g^{-1}}(\kappa(P_A(g)))\nempty\}.$$
In fact, $\Or(\V_n)$ is a factor of $\Or'(\U_n)$ by the factor map $\kappa^G:\Or'(\U_n)\To\Or(\V_n)$ induced by $\kappa$, that is, $\kappa^G(\{u_g\}_{g\in G}):=\{\kappa(u_g)\}_{g\in G}$.

The connection of hitting time set between $\Or'(\U_n)$ and $\Or(\V_n)$ is given as follows.
\begin{lemma}\label{l:6.14}
	For any finite subset $A\subset G$, and any $P_A,P'_A\in\mathcal{L}_A(\Or'(\U_n))$, we have $$N(C(P_A)\cap\Or'(\U_n),C(P'_A)\cap\Or'(\U_n))=N(C(Q_A)\cap\Or(\V_n),C(Q'_A)\cap\Or(\V_n))$$
	where $Q_A,Q'_A\in\mathcal{L}_A(\Or(\V_n))$,
	$Q_A=\kappa\circ P_A$ and $Q'_A=\kappa\circ P'_A$.
\end{lemma}

\begin{proof}
	Let $Q_A,Q'_A$ be defined as in the lemma.
	Since $\bigcap_{g\in A}\Phi_{g^{-1}}(\kappa(P_A(g)))\nempty$,
	$Q_A\in\mathcal{L}_A(\Or(\V_n))$.
	Also, $Q'_A\in\mathcal{L}_A(\Or(\V_n))$.
	
	Fix any $h\in N(C(P_A)\cap\Or'(\U_n),C(P'_A)\cap\Or'(\U_n))$.
	Then there exists $u=\{u_g\}_{g\in G}\in\Or'(\U_n)$ such that $u\in C(P_A)\cap\Or'(\U_n)$ and $\sigma_h(u)\in C(P'_A)\cap\Or'(\U_n)$.
	For some $a,a'\in A$ with $a=a'h$, notice that $P_A(a)=P'_A(a')$ since $\pi_a(u)=\pi_{a'h}(u)=\pi_{a'}(\sigma_h(u))$.
	Thus $$\bigcap_{a\in A}\Phi_{a^{-1}}(\kappa(P_A(a)))\cap\Phi_{{(ah)}^{-1}}(\kappa(P'_A(a)))\nempty.$$
	Let $v=\kappa^G(u)$. So $v\in C(Q_A)\cap\Or(\V_n)$ and $\sigma_hv\in C(Q'_A)\cap\Or(\V_n)$, which implies that $h\in N(C(Q_A)\cap\Or(\V_n),C(Q'_A)\cap\Or(\V_n))$.
	The converse is similar.
\end{proof}

Now, we will define a bonding map $\iota'$ to construct an inverse system which consists of $\Or'(\U_n)$.
Notice that $\kappa(\iota^{n+k}_n(u))\supset \kappa(u)$ for any $u\in\U_{n+k}$.
Then for any $\{u_g\}_{g\in G}\in\Or''(\U_{n+k})$, we have
$$\bigcap_{g\in G}\Phi_{g^{-1}}(\kappa(\iota^{n+k}_n(u_g)))\supset\bigcap_{g\in G}\Phi_{g^{-1}}(\kappa(u_g))\nempty,$$
which implies that $\{\iota^{n+k}_n(u_g)\}_{g\in G}\in\Or''(\U_n)$.
Thus the map $\iota^{n+k}_n$ induces the map $\iota':\Or'(\U_{n+k})\To\Or'(\U_n)$ by $$\iota'(\{u_g\}_{g\in G})=\{\iota^{n+k}_n(u_g)\}_{g\in G}.$$
It is clear that $\iota'$ is continuous and commutes with $\sigma$.
\begin{lemma}
	For any $n,k\in\N$, the map $\iota':\Or'(\U_{n+k})\To\Or'(\U_n)$ is surjective.
\end{lemma}

\begin{proof}
	Without loss of generality, we only prove that for any $\{u^n_g\}_{g\in G}\in\Or''(\U_n)$, there exists $\{u^{n+1}_g\}_{g\in G}\in\Or''(\U_{n+1})$ such that $\iota'(\{u^{n+1}_g\}_{g\in G})=\{u^n_g\}_{g\in G}$.
	
	Fix any $\{u^n_g\}_{g\in G}\in\Or''(\U_n)$.
	Take $x\in\bigcap_{g\in G}\Phi_{g^{-1}}(\kappa(u^n_g))$.
	Then by the construction of $\{\V_n\}$, for each $g\in G$, there exists $V_g\in\V_{n+1}$ with $V_g\subset\kappa(u^n_g)$ such that $\Phi_g(x)\in V_g$. Let $u^{n+1}_g\in\U_{n+1}$ satisfy $\iota^{n+1}_n(u^{n+1}_g)=u^n_g$ and $\kappa(u^{n+1}_g)=V_g$. Since $x\in\bigcap_{g\in G}\Phi_{g^{-1}}(V_g)$, we have $\{u^{n+1}_g\}_{g\in G}\in\Or''(\U_{n+1})$. And since $\iota^{n+1}_n(u^{n+1}_g)=u^n_g$, it is clear that $\iota'(\{u^{n+1}_g\}_{g\in G})=\{u^n_g\}_{g\in G}$, which ends the proof.
\end{proof}

So we can conclude that $(\iota',(\Or'(\U_n),G,\sigma))$ is an ML inverse system.
Thus we give the proof of Theorem \ref{t:1.4}.


\begin{proof}[Proof of Theorem \ref{t:1.4}]
	Let $\V_n$ and $\U_n$ be defined as above.
	By Lemma \ref{l:6.7} and Lemma \ref{l:6.14}, $(\Or'(\U_n),G,\sigma)$ also has property $P$.
	Therefore, by Lemma \ref{l:6.1}, the inverse limit of $(\iota',(\Or'(\U_n),G,\sigma))$ also has property $P$.
	Define $\phi:\ilim\{\iota',\Or'(\U_n)\}\To X$ satisfy
	$$\{\phi(\{u_i\}_{i\in\N})\}=\bigcap_{i\in\N}\overline{\pi_{e_G}(\kappa(u_i))}.$$
	Notice that $\pi_{e_G}(\kappa(u_{i+1}))\subset \pi_{e_G}(\kappa(u_i))$ and $\diam(\U_i)\To 0$, which implies that $\phi$ is well defined.
	It can be seen that $\phi$ is continuous, and $\phi\circ\sigma^*_g=\Phi_g\circ\phi$.
	
	We claim that $\phi$ is surjective.
	For any $x\in X$, let
	$$O_i(x)=\{\{u_g\}_{g\in G}\in\Or'(\U_i):x\in\bigcap_{g\in G}\Phi_{g^{-1}}(\kappa(u_g))\}$$
	Let $O(x)=\ilim\{\iota',\Or'(\U_n)\}\cap\bigcap_{i\in\N}\pi_i^{-1}\left(\overline{O_i(x)}\right)$.
	Thus $O(x)$ is nonempty, and $\phi(O(x))=\{x\}$, which implies that $\phi$ is surjective.
	Therefore, $(X,G,\Phi)$ is a factor of the inverse limit of an ML inverse system of subshifts with property $P$.
	
	For the converse, denote by $(X^*,G,\sigma^*)$ the inverse limit in theorem, and by Lemma \ref{l:6.6}, $(X^*,G,\sigma^*)$ has property $P$.
	For a factor map $\phi$ from $X^*$ to $X$, it is clear that $N(U,V)=N(\phi^{-1}(U),\phi^{-1}(V))$ for any two nonempty open subsets $U,V\subset X$.
	Thus $(X,G,\Phi)$ has property $P$.
\end{proof}

\begin{remark}
	Similar with Theorem \ref{t:1.3}, Theorem \ref{t:1.4} holds when property $P$ is $\mathcal{F}$-transitive, where  $\mathcal{F}\subset\mathcal{F}_{\mathrm{inf}}$ is a Furstenberg family satisfying that $F\setminus A\in\mathcal{F}$ for any $F\in\mathcal{F}$ and any finite subset $A\subset G$.
\end{remark}

Similar to Corollary \ref{c:6.13}, we have the following corollary.

\begin{corollary}
	Suppose that $f:X\To X$ is a continuous surjection, where $X$ is a compact metric space. Let property $P$ be one of the following properties: transitivity, totally transitivity, weakly mixing, mixing. Then $f$ has property $P$ if and only if it is conjugate to the inverse limit of an ML inverse system of subshifts with property $P$.
\end{corollary}


\section*{Acknowledgements}
The second author was supported by NNSF of China (11671208 and 11431012). We would like to express our gratitude to Tianyuan Mathematical Center in Southwest China, Sichuan University and Southwest Jiaotong University for their support and hospitality.

\end{document}